\documentclass[]{article}
\usepackage{epsfig,amsfonts,amsmath,amssymb,amscd,float}
\usepackage[dvipsnames]{xcolor}
\usepackage{graphicx,caption,float,subcaption}
\usepackage{booktabs,multirow,siunitx}
\usepackage{amsthm}
\usepackage{amsmath}
\usepackage{parskip}
\usepackage{cases}
\usepackage{verbatim}
\usepackage{authblk}
\usepackage{algorithm}
\usepackage{algpseudocode}
\usepackage{colortbl,hhline}
\newtheorem{theorem}{Theorem}[section]
\newtheorem{lemma}[theorem]{Lemma}

\newtheorem{proposition}[theorem]{Proposition}
\theoremstyle{remark}
\newtheorem{remark}{ Remark}
\theoremstyle{definition}
\newtheorem{definition}{Definition}
\numberwithin{equation}{section}
\newcommand{\cred}[1]{{\color{red}  #1}}

\newcommand{\R}{\mathbb{R}}
\renewcommand{\v}[1]{\ensuremath{\mathbf{#1}}}

\usepackage{ulem}
\begin{document}
	\title{\bf A note on the convergence of multigrid methods for the Riesz-space equation and an application to image deblurring}
	\author{
  Danyal Ahmad\textsuperscript{a}\thanks{dahmad@uninsubria.it (D. Ahmad)}\,\,, Marco Donatelli\textsuperscript{a}\thanks{marco.donatelli@uninsubria.it (M. Donatelli)}\,\,, Mariarosa Mazza\textsuperscript{b}\thanks{mariarosa.mazza@uniroma2.it (M. Mazza)}\,\,, Stefano Serra-Capizzano\textsuperscript{a}\thanks{s.serracapizzano@uninsubria.it (S. Serra-Capizzano)}\,\,, Ken~Trotti\textsuperscript{c}\thanks{ken.trotti@usi.ch (K. Trotti)} 
    }
 \affil{\textsuperscript{a} Dipartimento di Scienza e Alta Tecnologia, Universit\`a dell'Insubria, Como, Via Valleggio~11, 22100 Como, Italy}
 \affil{\textsuperscript{b} Dipartimento di Matematica, Università degli studi di Roma Tor Vergata, Via della Ricerca Scientifica~1, 00133 Rome, Italy}
 \affil{\textsuperscript{c} Facoltà di informatica, Universit\`a della Svizzera Italiana, Via Giuseppe Buffi~13, CH-6900 Lugano, Switzerland}
	\date{}
	\maketitle
	
	\numberwithin{equation}{section}

\begin{abstract}
In the past decades, a remarkable amount of research has been carried out regarding fast solvers for large linear systems resulting from various discretizations of fractional differential equations (FDEs).
In the current work, we focus on multigrid methods for a Riesz-space FDE whose theoretical convergence analysis of such multigrids is currently limited to the two-grid method. Here we provide a detailed theoretical convergence study in the case of V-cycle and W-cycle. Moreover, we discuss its use combined with a band approximation and we compare the result with both $\tau$ and circulant preconditionings. The numerical tests include 2D problems as well as the extension to the case of a Riesz-FDE with variable coefficients. Finally, we apply the best-performing method to an image deblurring problem with Tikhonov regularization.
\end{abstract}
	
\section{Introduction}

In the present work, we are interested in the fast numerical solution of special large linear systems stemming from the approximation of a Riesz fractional diffusion equation (RFDE). It is known that using standard finite difference schemes on equispaced gridding, in the presence of one-dimensional problems we end up with a dense real symmetric Toeplitz linear system, whose dense structure is inherited from the non-locality of the underlying symmetric operator.
We study a one-dimensional RFDE discretized with a numerical scheme with precision order one mainly for the simplicity of the presentation. Nevertheless, the same analysis can be carried out either when using higher-order discretization schemes \cite{high-order-1,mazza2023matrices} or when dealing with more complex higher-dimensional fractional models \cite{meerschaert2004finite,complex-1}.

While the matrix-vector product involving dense Toeplitz matrices has a cost of $O(M\log M)$ with a moderate constant in the ``big O" and with $M$ being the matrix size, the solution can be of higher complexity in the presence of ill-conditioning. In connection with the last issue, it is worth recalling that the Toeplitz structure carries a lot of spectral information (see \cite{GLT_I_book,GLT_Eng} and references therein). As shown in \cite{donatelli2016spectral,mazza2023matrices} the Euclidean conditioning of the coefficient matrices grows exactly as $M^\alpha$.
Therefore the classical conjugate gradient (CG) method would need a number of iterations exactly proportional to $M^{\alpha/2}$ to converge within a given precision.

In this perspective, a large number of preconditioning strategies for the preconditioned CG (PCG) have been proposed. One possibility is given by band-approximations of the coefficient matrices, which as discussed in \cite{donatelli2016spectral,donatelli2020multigrid} does not ensure a convergence within a number of iterations independent of $M$, but still prove to be efficient for some values of the fractional order and whose cost stays linear within the matrix-size. As an alternative, the circulant proposal given in \cite{lei2013circulant} has a cost proportional to the matrix-vector product and ensures convergence in the presence of 1D problems that do not include variable coefficients. Another algebraic preconditioner proven to be optimal consists in the $\tau$ preconditioning explored in \cite{NMTMA,BEV}. In a multi-dimensional setting, the multilevel circulants are not well-suited \cite{capizzano2000any}, while the $\tau$ preconditioning still shows optimal behavior. For other effective preconditioners, we refer the reader to \cite{ACETO2023372} where the proposed preconditioner is based on a rational approximation of the Riesz operator.

Other approaches to the efficient solution of the ill-conditioned linear systems we are interested in include multigrid methods. Optimal multigrid proposals for fractional operators are given e.g. in \cite{pang2012multigrid,moghaderi2017spectral,DMS18,donatelli2020multigrid}. However, a theoretical convergence analysis of such multigrids is currently limited to the two-grid method. Here we provide a detailed theoretical convergence study in the case of V-cycle and W-cycle. Moreover, as multigrid methods are often applied as preconditioners themselves, and often to an approximation of the coefficient matrix instead of to the coefficient matrix itself, on the same line of what has been done in \cite{donatelli2020multigrid}, we discuss a band approximation to be solved with a Galerkin approach which guarantees linear cost in $M$ instead of $O(M\log M)$ and that inherits the V-cycle optimality. We compare the resulting methods with both $\tau$ and circulant preconditioning and apply the best performing strategy to an image deblurring problem with Tikhonov regularization.

The paper is organized as follows. In Section \ref{sec:setting} we present the problem, the chosen numerical approximation, and the resulting linear systems. The main spectral features of the considered coefficient matrices are also recalled. In Section \ref{sec:multigrid} we describe the essentials of multigrid methods and give the description of our proposal and its analysis, while in Section \ref{sec:band_approximation} we discuss its use applied to a band approximation of the coefficient matrix. Section~\ref{sec:num} contains a wide set of numerical experiments that compare the multigrid with state-of-the-art preconditioners and include 2D problems as well as the extension to the case of an RFDE with variable coefficients. Section~\ref{sec:imdeblur} is devoted to applying the best preconditioning strategy to an image deblurring problem. Future research directions and conclusions are given in Section \ref{sec:end}.

%---------------------------------------------------------------
\section{Preliminaries}\label{sec:setting}

The current section briefly introduces the continuous problem, its numerical approximation scheme, and the resulting linear systems. The main spectral features of the considered coefficient matrices are recalled as well.

As anticipated in the introduction, we consider the following one-dimensional Riesz Fractional Diffusion Equation (1D-RFDE) given by
\begin{align}\label{eq:RFDE}
			-d\frac{\partial^{\alpha} u(\,x)\,}{\partial\vert x\vert^{\alpha}} =m(\,x)\,,\quad x\in\Omega=[a,b]\,
	\end{align}
	with boundary conditions
	\begin{align}\label{eq:BC_RFDE}
			u(\,x)\ =0,\quad x\in\partial\Omega,
	\end{align}
where $d>0$ is the diffusion coefficient, $m(x)$ is the source term, and \begin{math}\frac{\partial^{\alpha} u(\,x)\,}{\partial\vert x\vert^{\alpha}}\end{math} is the  1D-Riesz fractional derivative for $\alpha\in(\,1,2)$, defined as
	\begin{align}\label{eq:RFDE_LR}
\frac{\partial^{\alpha} u(\,x)\,}{\partial\vert x\vert^{\alpha}}=c(\alpha)({_{a}{{{D}}_{x}^{\alpha}}u(\,x)\,}+{_{x}{{{D}}_{b}^{\alpha}}u(\,x)\,}),			
	\end{align}
with $c(\,\alpha)\,=-\frac{1}{2\cos{\frac{\alpha\pi}{2}}}>0$, while ${_{a}{\mathcal{D}_{x}^{\alpha}}u(\,x)\,}$ and ${_{x}{\mathcal{D}_{b}^{\alpha}}u(\,x)\,}$ are the left and right Riemann-Liouville fractional derivatives.
	
For the approximation of the problem, we define the uniform space partition on the interval $[\,a,b]$, that is
	\begin{equation}\label{eq:partition}
	 x_{i} = a+ih,\quad h=\frac{b-a}{M+1},\quad i=1,\dots,M.
	\end{equation}
Here the left and right fractional derivatives in \eqref{eq:RFDE_LR} are numerically approximated by the shifted Gr\"unwald difference formula as
	\begin{align}\label{eq:LR_WSGD}
	\begin{split}
		{_{a}{{D}_{x}^{\alpha}}u(x_{i})}&=\frac{1}{h^{\alpha}}\sum_{k=0}^{i+1}g_{k}^{(\,\alpha)\,}u(\,x_{i-k+1})\,+O{(\,h)\,},\\
		{_{x}{{D}_{b}^{\alpha}}u(x_{i})}&=\frac{1}{h^{\alpha}}\sum_{k=0}^{M-i+2}g_{k}^{(\,\alpha)\,}u(\,x_{i+k-1})\,+O{(\,h)\,},
	\end{split}
	\end{align}
	with $g_{k}^{(\,\alpha)\,}$ being the alternating fractional binomial coefficients, whose expression is
	\begin{equation}\label{eq:g_Coef}
	    g_{k}^{(\,\alpha)\,}=(-1)^{k}\binom{\alpha}{k}=\frac{(-1)^{k}}{k!}{\alpha(\,\alpha-1)\,\cdots(\,\alpha-k+1)\,}, \quad k\in \mathbb{N}_0.
	\end{equation}
	By employing the approximation in \eqref{eq:LR_WSGD} to equation in \eqref{eq:RFDE} and using \eqref{eq:RFDE_LR} and \eqref{eq:g_Coef}, we obtain the required scheme as follows
	\begin{align}\label{eq:scheme}
			-d\frac{c(\alpha)}{h^{\alpha}}\Big(\sum_{k=0}^{i+1}g_{k}^{(\,\alpha)\,}u_{i-k+1}\,+\sum_{k=0}^{M-i+2}g_{k}^{(\,\alpha)\,}u_{i+k-1}\Big)=m_{i},\quad i=1,\dots,M.
	\end{align}
	By defining ${\bf u}=[\,u_{1},u_{2},\dots,u_{M}]\,^{T}$ and ${\bf m}=[\,m_{1},m_{2},\dots,m_{M}]\,^{T}$, the resulting approximation equations \eqref{eq:scheme} can be written in matrix form as
		\begin{align*}
		\overline{c}{({\mathcal{G}^{\alpha}_{M}}+{{\mathcal{G}^{\alpha}_{M}}}^{T})}{{\bf u}}={\bf m}
	\end{align*}
and, by defining $\mathcal{\bf G}_{M}^{\left(\alpha\right)}={\mathcal{G}^{\alpha}_{M}}+{{\mathcal{G}^{\alpha}_{M}}}^{T}$, we have
		\begin{align}\label{eq:Final_Linear_sys}
		{A^{\alpha}_{M}}{{\bf u}}={\bf m},
	\end{align}
	with ${A^{\alpha}_{M}}=\overline{c}\,\mathcal{\bf G}_{M}^{\left(\alpha\right)}$, 
	$\overline{c}=-\frac{dc(\alpha)}{h^{\alpha}}$, and
	\begin{equation}\label{eq:LH_matrix}
    {\mathcal{G}^{\alpha}_{M}}=
	\begin{bmatrix}
		& {g_{1}^{(\alpha)\,}} & {g_{0}^{(\alpha)\,}} & 0 & \cdots & 0 & 0 &\\
		& {g_{2}^{(\alpha)\,}} & {g_{1}^{(\alpha)\,}} & {g_{0}^{(\alpha)\,}} & \cdots & 0 & 0 &\\
		& \vdots & \ddots & \ddots & \ddots &\ddots & \vdots&\\
		& \vdots & \ddots & \ddots &\ddots &\ddots & \vdots&\\
		& {g_{M-1}^{(\alpha)\,}} & {g_{M-2}^{(\alpha)\,}} & \cdots & \cdots & {g_{1}^{(\alpha)\,}} & {g_{0}^{(\alpha)\,}} &\\
		& {g_{M}^{(\alpha)\,}} & {g_{M-1}^{(\alpha)\,}} &\cdots & \cdots & {g_{2}^{(\alpha)\,}} & {g_{1}^{(\alpha)\,}}&
	\end{bmatrix}_{{M\times M}}
\end{equation}
being a lower Hessenberg Toeplitz matrix. The fractional binomial ${g_{k}^{(\alpha)\,}}$ coefficients satisfy few properties summarised in the following Lemma \ref{Lemma1} (see \cite{meerschaert2004finite,meerschaert2006finite,wang2010direct}).
	\begin{lemma}\label{Lemma1} For $1<\alpha<2$, the coefficients $g_{k}^{(\alpha)\,}$ defined in \eqref{eq:g_Coef} satisfy
		\begin{equation}
			\begin{cases}
			        g_{0}^{(\alpha)\,}=1,\quad g_{1}^{(\alpha)\,}=-\alpha,\quad g_{0}^{(\alpha)\,}>g_{2}^{(\alpha)\,}>g_{3}^{(\alpha)\,}>\dots>0,\\
    		        \sum_{k=0}^{\infty}{g_{k}^{(\alpha)\,}}=0,\quad \sum_{k=0}^{n}{g_{k}^{(\alpha)\,}}<0,\quad \mbox{for} \quad n\geq 1.
			\end{cases}
		\end{equation}
\end{lemma}
Using Lemma \ref{Lemma1}, it has been proven in \cite{wang2010direct} that the coefficient matrix ${{A^{\alpha}_{M}}}$ is strictly diagonally dominant and hence invertible. To determine its generating function and study its spectral distribution, tools on Toeplitz matrices have been used (see \cite{donatelli2016spectral}). A short review of these results follows.

\begin{definition}\label{def1_fourier_coefficients}
    \cite{GLT_I_book}
Given $f\in L^1(-\pi,\pi)$ and its Fourier coefficients
\begin{equation}\label{eq:fourier_coef}
		a_{k}=\frac{1}{2\pi}\int_{-\pi}^{\pi}{f(x){\rm{e}}^{-\iota kx}dx},\quad\iota^{2}=-1,~~k\in\mathbb{Z},
	\end{equation}
 we define $T_{M}(f)\in\mathbb{C}^{M\times M}$ the Toeplitz matrix
	\begin{equation}\label{eq:Toeplitz_M}
	    T_{M}(f)=\begin{bmatrix}
		& a_{0} & a_{-1} & a_{-2} & \cdots & \cdots & a_{1-M} &\\
		& a_{1} & a_{0} & a_{-1} & \cdots & \cdots & a_{2-M} &\\
		& \vdots & \ddots & \ddots & \ddots &\ddots & \vdots&\\
		& \vdots & \ddots & \ddots &\ddots &\ddots & \vdots&\\
		& a_{M-2} & \ddots & \ddots & \ddots & \ddots & a_{-1} &\\
		& a_{M-1} & a_{M-2} &\cdots & \cdots & a_{1} & a_{0} &
	\end{bmatrix}
	\end{equation}
generated by $f$. In addition, the Toeplitz sequence $\left\{{T_{M}(f)}\right\}_{M\in \mathbb{N}}$ is called the family of Toeplitz matrices generated by $f$.
 \end{definition}

For instance, the tridiagonal Toeplitz matrix
\[
	   \mathcal{G}^2_{M}=\begin{bmatrix}
		& 2 & -1 & &  & &\\
		& -1 & 2 & -1 &  & &\\
		&  & \ddots & \ddots & \ddots & &\\
		& &  &-1 &2 & -1&\\
		& &  &  & -1 & 2 &
	\end{bmatrix}
\]
associated with a finite difference discretization of the Laplacian operator is generated by
\begin{equation}\label{eq:symbl}
l(x)	= 2-2\cos(x).
\end{equation}

 Note that for Toeplitz matrices $T_{M}(f)$, $M\in \mathbb{N}$, as in \eqref{eq:Toeplitz_M}, to have a generating function associated to the whole Toeplitz sequence, we need that there exists $f\in L^1(-\pi,\pi)$ for which the relationship (\ref{eq:fourier_coef}) holds for every $k\in\mathbb{Z}$.

 In the case where the partial Fourier sum
	\[
	\sum_{k=-M+1}^{M-1}{a_{k}{\rm{e}}^{\iota kx}}
	\]
	converges in infinity norm, $f$ coincides with its sum and is a continuous $2\pi$ periodic function, given the Banach structure of this space equipped with the infinity norm. A sufficient condition is that $\sum_{k=-\infty}^{\infty}\vert{a_{k}}\vert<\infty$, i.e., the generating function belongs to the Wiener class, which is a closed sub-algebra of the continuous $2\pi$ periodic functions.

 Now, according to \eqref{eq:LH_matrix}, we define
	\begin{equation*}%\label{eq:generating_ftn_Ba}
		b_{M,\alpha}\left(x\right)=\sum_{k=0}^{M}{g_{k}^{(\alpha)}{\rm{e}}^{\iota (k-1)x}}.
	\end{equation*}
Hence the partial Fourier sum associated to  $\mathcal{\bf G}_{M}^{\left(\alpha\right)}$ is
 \begin{equation*}%\label{eq:generating_ftn}
%		f_{M,\alpha}\left(x\right)=
  b_{M,\alpha}\left(x\right)+\overline{	b_{M,\alpha}\left(x\right)}=2
		\sum_{k=0}^{M}{g_{k}^{(\alpha)}\cos\left((k-1)x\right)}.
	\end{equation*}

	Fixed $d=1$, from \eqref{eq:LH_matrix} it is evident that the matrix $h^\alpha A_M^\alpha=c(\alpha)\mathcal{\bf G}_{M}^{\left(\alpha\right)}$ is symmetric, Toeplitz, and is generated by the following real-valued even function on $[-\pi,\pi]$, defined in the following lemma.
	
\begin{lemma}\label{Lemma: symbol}
{\rm \cite{pang2016fast,donatelli2016spectral}}  The generating function of $h^\alpha A_M^\alpha=c(\alpha)\mathcal{\bf G}_{M}^{\left(\alpha\right)}$ is
		\begin{equation}\label{eq:symbol_ftn}
  f_{\alpha}(\,x)\,=-c(\alpha)2^{\alpha+1}\Big(\sin{\frac{x}{2}}\Big)^{\alpha}\cos{\Big(\,x\Big(\,1-\frac{\alpha}{2}\Big)\,+\frac{\alpha\pi}{2}\Big)\,},\quad~~ x\in [\,0,\pi]\,.
		\end{equation}
\end{lemma}
\begin{itemize}

        \item For $\alpha=2$,\quad $f_{2}(\,x)\,=l(\,x\,)\,=2-2\cos{x}$.
        \item For $\alpha=1$,\quad $f_{1}(\,x)\,=\frac{1}{2}l(\,x\,)$
    \end{itemize}
As observed in \cite{donatelli2016spectral}, the function $f_{\alpha}(\,x)\,$ has a unique zero of order $\alpha$ at $x=0$ for $1<\alpha\leq 2$.

By using the extremal spectral properties of Toeplitz matrix sequences (see e.g. \cite{Se96,Se98,BGr98}) and by combining it with Lemma \ref{Lemma: symbol}, for $\alpha \in (1,2)$, the obtained spectral information can be summarized in the following items  (see again \cite{donatelli2016spectral}):
\begin{itemize}
\item any eigenvalue $\lambda$ of $h^\alpha A_M^\alpha$ belongs to the open interval $(0,r)$ with $0=\min f_{\alpha}$, $r=\max f_{\alpha}>0$;
\item $\lambda_{\max}(h^\alpha A_M^\alpha)$ is monotonic strictly increasing converging to $r=\max f_{\alpha}$, as $M$ tends to infinity;
\item $\lambda_{\min}(h^\alpha A_M^\alpha)$ is monotonic strictly decreasing converging to $0=\min f_{\alpha}$, as $M$ tends to infinity;
\item  by taking into account that the zero of the generating function $f_{\alpha}$ has exactly order $\alpha\in (1,2]$, it holds that
\[
\lambda_{\min}(h^\alpha A_M^\alpha) \sim \frac{1}{M^\alpha},
\]
where, for nonnegative sequences $\alpha_n, \beta_n$, the writing $\alpha_n\sim \beta_n$ means that there exist real positive constants $0<c\le C< \infty$, independent of $n$ such that
$
c \alpha_n \le \beta_n \le C \alpha_n, \forall n \in {\mathbb N}.
$
\end{itemize}
As a consequence of the third and of the fourth items reported above, the Euclidean conditioning of $h^\alpha A_M^\alpha$ grows exactly as $M^\alpha$.

Therefore, classical stationary methods fail to be optimal, in the sense that we cannot expect convergence within a given precision, with a number of iterations bounded by a constant independent of $M$: for instance, the Gauss-Seidel iteration would need a number of iterations exactly proportional to $M^\alpha$. In addition, even the Krylov methods will suffer from the latter asymptotic ill-conditioning. In reality, looking at the estimates by Axelsson and Lindsk\"og \cite{AxL}, $O(M^{\alpha/2})$ iterations would be needed by a standard CG for reaching a solution within a fixed accuracy. In the present context of a zero of noninteger order, even the preconditioning by using band-Toeplitz matrices, will not ensure optimality as discussed in \cite{donatelli2016spectral} and detailed in Section \ref{sec:band_approximation}. On the other hand, when using circulant or the $\tau$ preconditioning, the spectral equivalence is guaranteed also for a noninteger order of the zero at $x=0$ if the order is bounded by $2$ \cite{lei2013circulant,NMTMA,BEV}. In the latter case, the distance in rank is less with respect to other matrix-algebras as emphasized in \cite{Slinear}. In a multi-dimensional setting, only the $\tau$ preconditioning leads to optimality using the PCG under the same constraint on the order of the zero, while multilevel circulants are not well-suited \cite{capizzano2000any}. Finally, since we consider also problems in imaging we recall that the $\tau$ algebra emerges also in this context when using high precision boundary conditions (BCs) such as the anti-reflective BCs \cite{tau-AR-BCs-1,tau-AR-BCs-2}. %For other effective preconditioners based on rational approximations of the Riesz operator we refer the reader to \cite{ACETO2023372}.

%---------------------------------------------------------------
 \section{Multigrid methods and level independency}\label{sec:multigrid}

 Multigrid methods (MGMs) have already proven to be effective solvers as well as valid preconditioners for Krylov methods when numerically approaching $\mbox{FDEs}$ \cite{moghaderi2017spectral,donatelli2020multigrid,pang2012multigrid}. A multigrid method combines two iterative methods called smoother and coarse grid correction. The two basic procedures when carefully chosen are spectrally complementary and in this sense, multigrid procedures are always of multi-iterative type \cite{CMAME-colloc}.   The coarser matrices can either be obtained by projection (Galerkin approach) or by rediscretization (geometric approach). In the case where only a single coarser level is considered we obtain Two Grids Methods (TGMs), while in the presence of more levels, we have multigrid methods, in particular, V-cycle when only a recursive call is performed and W-cycle when two recursive calls are applied.

 Deepening what has been done in \cite{donatelli2016spectral,moghaderi2017spectral}, in this section, we investigate the convergence of multigrid methods in Galerkin form for the discretized problem~\eqref{eq:Final_Linear_sys}.

 %The coefficient matrix ${A^{\alpha}_{M}}$ is a symmetric positive definite Toeplitz matrix, so damped Jacobi as a smoother is a good choice \cite{sun1997note}.
To apply the Galerkin multigrid to a given linear system $A_Mx=b$, for $M_0 = M > M_1 > M_2 > \dots > M_\ell > 0$, let us define the grid transfer operators $P_k \in \R^{M_{k+1}\times M_k}$ as
 \begin{equation}\label{eq:Pk}
 %{eq:gal}
   P_k = K_kT_{M_k}(p_{k}), \qquad k=0,\dots, \ell-1,
 \end{equation}
where $K_k$ is the down-sampling operator and $p_{k}$ is a properly chosen polynomial. The coarser matrices are then
 \begin{equation}\label{eq:gal}
  A_{M_{k+1}}= P_kA_{M_{k}} P_k^T, \qquad k=0,\dots, \ell-1.
 \end{equation}

%Following the analysis in~\cite{chan1998multigrid}, given a symmetric positive definite matrix ${A^{\alpha}_{M}}$, we call $D$ the diagonal of ${A^{\alpha}_{M}}$.
If at the recursive level $k$, we simply apply one post-smoothing iteration of a stationary method having iteration matrix $S_k$ the TGM iteration matrix at the level $k$ is given by
\begin{equation*}\label{eq:TGM}
    TGM_k=S_k\left[I_{M_{k}}-P^T_k(P_k{A_{M_{k}}}P^T_k)^{-1}P_k{A_{M_{k}}}\right].
\end{equation*}

%Thanks to the symmetric positive definite property of the matrix ${A^{\alpha}_{M_{k}}}$, and defining $D_k$ as the diagonal matrix having as main diagonal that of ${A^{\alpha}_{M_{k}}}$.%, we can define the following inner products:
%\begin{equation*}\label{eq:inner_Pro}
 %   \langle u_{1},u_{2}\rangle_{0}=\langle Du_{1},u_{2}\rangle,\quad \langle u_{1},u_{2}\rangle_{1}=\langle {A^{\alpha}_{M_k}}u_{1},u_{2}\rangle,\quad \langle u_{1},u_{2}\rangle_{2}=\langle D_{k}^{-1}{A^{\alpha}_{M_k}}u_{1},{A^{\alpha}_{M_{k}}}u_{2}\rangle,
%\end{equation*}
%where $\langle \cdot,\cdot\rangle$ is the Euclidean inner product, and their respective norms $\vert\vert\cdot\vert\vert_{j}$, for $j=0,1,2$.

The following theorem states the convergence of TGM at a generic recursion level $k$.
\begin{theorem}[Ruge-St\"uben \cite{ruge1987algebraic}]\label{theorem4}

Let ${A_{M_k}}$ be a symmetric positive definite matrix, $D_k$ the diagonal matrix having as main diagonal that of ${A_{M_{k}}}$, and $S_k$ be the post-smoothing iteration matrix. Assume that $\exists\ \sigma_{k}>0$ such that
\begin{equation}\label{eq:Smoothing_pro}
\vert\vert S_k u\vert\vert_{A_{M_{k}}}^{2}\leq\vert\vert u\vert\vert_{A_{M_{k}}}^{2}-\sigma_{k}\vert\vert u\vert\vert_{{A_{M_{k}}}D_{k}^{-1}{A_{M_{k}}}}^{2},\quad\forall~u\in\mathbb{R}^{M_k},
\end{equation}
%%%%%%%
and $\exists\ \delta_{k}>0$ such that
\begin{equation}\label{eq:App_Pro}
\min_{y\in\mathbb{R}^{M_{k+1}}} \vert\vert u-P_{k} y\vert\vert_{D_{k}}^{2}\leq\delta_{k}\vert\vert u\vert\vert_{A_{M_{k}}}^{2},\quad\forall~u\in\mathbb{R}^{M_k},
\end{equation}
%%%%%%
then $\delta_{k}>\sigma_{k}$ and
\begin{equation*}\label{eq:norm_TGM}
    \vert\vert \mbox{TGM}_{k}\vert\vert_{A_{M_{k}}}\leq\sqrt{1-\frac{\sigma_{k}}{\delta_{k}}}.
\end{equation*}
\end{theorem}
The inequalities in equations \eqref{eq:Smoothing_pro} and \eqref{eq:App_Pro} are well known as the \emph{smoothing property} and \emph{approximation property}, respectively. To prove the multigrid convergence (recursive application of TGM) it is enough to prove that the assumptions of Theorem \ref{theorem4} are satisfied for each $k=0,1,\dots,\ell-1$.
On the other hand, to guarantee a W-cycle convergence in a number of iterations independent of the size $M$ \cite{Trot}, we have to prove that the following \emph{level independency} condition
\begin{equation}\label{eq:levind}
    \liminf_{k\to \infty} \, \frac{\sigma_k}{\delta_k} \, = \, \beta > 0
\end{equation}
holds with $\beta$ independent of $k$.

Concerning (multilevel) Toeplitz matrices, multigrid methods have been extensively investigated in the literature \cite{FS2,chan1998multigrid,arico2007}.
While the smoothing property can be easily carried on for a simple smoother like weighted Jacobi, the approximation property usually requires a detailed analysis based on a generalization of the local Fourier analysis, see \cite{D10}.

If we now apply the multigrid with Galerkin approach to the solution of \eqref{eq:Final_Linear_sys} and refer to the matrices at the coarser level $k$ as $A^\alpha_{M_k}$, they are still Toeplitz and, based on the theoretical analysis developed in \cite{FS1,ADS} for $\tau$ matrices, are associated to the following symbols

  \begin{equation}\label{eq:f_at_level}
     f_{k+1,\alpha}(\,x)\,=\frac{1}{2}\Big[\,f_{k,\alpha}\Big(\,\frac{x}{2}\Big)\,p_{k,\alpha}^{2}\Big(\,\frac{x}{2}\Big)\,+f_{k,\alpha}\Big(\,\pi-\frac{x}{2}\Big)\,p_{k,\alpha}^{2}\Big(\,\pi-\frac{x}{2}\Big)\,\Big]\,,\qquad k=0,\dots, \ell-1.
  \end{equation}

Since the function $f_{0,\alpha}=f_\alpha$ has only a zero at the origin of order smaller or equal to 2, according to the convergence analysis of multigrid methods in \cite{FS1,ADS},
we define the symbols of the projectors $P_k$ in \eqref{eq:Pk} as
\begin{equation}\label{eq:symbpk}
  p_k(x)=p_{k,\alpha}(x), \qquad \mbox{ where }  \qquad p_{k,\alpha}(x)=C_{k+1,\alpha}p(x)=C_{k+1,\alpha}(1+\cos{x})
\end{equation}
and $C_{k+1,\alpha}$ is a constant chosen such that $f_{k+1,\alpha}(\,\pi)\,=f_{k,\alpha}(\,\pi)$.

%\cred{Add the smoothing analysis proving that the previous condition on $C_{k+1,\alpha}$ ensures that it is bounded at each level.}

We observe that the symbols $f_{k,\alpha}$ at each level are nonnegative and they vanish only at the origin with order $\alpha$ thanks to \cite[Proposition 2.3]{SerraNM}. Therefore, the symbols $p_{k,\alpha}$  in \eqref{eq:symbpk} and $f_{k+1,\alpha}$
in \eqref{eq:f_at_level} satisfy the approximation property \eqref{eq:App_Pro} because
\begin{equation}\label{eq:condTGMk}
    \lim_{x\to 0} \frac{p_{k,\alpha}(x+\pi)^2}{f_{k+1,\alpha}(x)}=\gamma_k<\infty
\end{equation}
at each level $k=0,\dots, \ell-1$, see \cite[Proposition 3.1]{SerraNM}.

\begin{remark}\label{rem:gamma}
Since $\delta_k$ is proportional to $\gamma_k$, the level independency condition \eqref{eq:levind} requires that
\begin{equation}\label{eq:levindToep}
    \limsup_{k\to \infty} \gamma_k=\gamma <\infty
\end{equation}
with $\gamma$ independent of $k$.
\end{remark}

For the smoothing property \eqref{eq:Smoothing_pro}, in order to ensure $\sigma_k>\sigma>0$ for all $k$, will be crucial the choice of
$C_{k+1,\alpha}$ as a constant such that $f_{k+1,\alpha}(\,\pi)\,=f_{k,\alpha}(\,\pi)$.

In conclusion, to prove the level independency condition \eqref{eq:levind}, we need a detailed analysis of the symbol
$f_{k+1,\alpha}$, which is performed in the next subsection.

 \subsection{Behaviour of symbols $f_{k,\alpha}$}\label{ssec:symbol recursion}
%\texorpdfstring{$\boldsymbol{f_{k,\alpha}}$}{symbols}
%\cred{There is a problem with the sign of $f_{k+1,\alpha}$: it was negative, but from this section, we consider it positive.}

The analysis of $f_{k+1,\alpha}$
in \eqref{eq:f_at_level} requires a detailed study of the constant $C_{k+1,\alpha}$ in
$p_{k,\alpha}$ chosen imposing the condition
\[f_{k+1,\alpha}(\,\pi)\,=f_{k,\alpha}(\,\pi)=c(\alpha)2^{\alpha+1}, \qquad k=0,1,\dots,\ell-1,\]
because $f_{0,\alpha}(\pi)=c(\alpha)2^{\alpha+1}$.

According to the classical analysis,
in the case of discrete Laplacian, i.e., $l(x)=2-2\cos{x}$, the value of $C_{k+1,2}$ at each level is $\sqrt{2}$.
%{\color{red}($\alpha=2$) THE CASE $\alpha=1$ DOES NOT MAKE SENSE FOR RIESZ, SO I HAD TO CHANGE THIS AND THE FOLLOWING},  the value of $C_{k+1,2}${\color{red} THIS IS THE CONSTANT FOR THE LAPLACIAN SO $C_{k+1,2}$ and not $C_{k+1,\alpha}$} at each level is $2\sqrt{2}$.

\begin{proposition}\label{proposition_f_k(x)=l(x)}
Let $f_{0,2}(x)\,=l(x)$, where $l(x)=2-2\cos{x}$, and $p_{k,2}(x)\,=C_{k+1,2}\,p(x)$~with $C_{k+1,2}\,=\sqrt{2}$. Then,
for $f_{k+1,2}$ computed according to \eqref{eq:f_at_level}, it holds
\begin{equation}\label{eq:f_k(x)=l(x)}
    f_{k+1,2}(x)\,=l(x), \qquad k=0,1,\dots,\ell-1.
\end{equation}
\end{proposition}
\begin{proof}
Since $p_{k,2}(x)=\sqrt{2}\,p(x)=\sqrt{2}\,(1+\cos{x})$,
using relation \eqref{eq:f_at_level},
for $k=0$ it holds
\[
  f_{1,2}(x)\,=l\Big(\,\frac{x}{2}\,\Big)\,p^{2}\Big(\,\frac{x}{2}\,\Big)\,+l\Big(\,\pi-\frac{x}{2}\,\Big)\,p^{2}\Big(\,\pi-\frac{x}{2}\,\Big)\,=l(\,x)\,.
\]
Since the symbol $p_{k,2}$ of the projector is the same at each level, by recursive relation we obtain the required result
$$f_{\ell,2}(x)\,=\ldots=\,f_{1,2}(x)\,=f_{0,2}(x)=l(x).$$
\end{proof}

For $\alpha\in(1,\,2]$, the value of $C_{k+1,\alpha}$, computed such that $f_{k+1,\alpha}(\,\pi)\,=f_{k,\alpha}(\,\pi)$, is
%For $k=0$, $C_{1,\alpha}=\sqrt{\frac{\sqrt{2}^\alpha}{\sin{\frac{\alpha\pi}{4}}}}$, such that $f_{1,\alpha}(\,\pi)\,=f_{0,\alpha}(\,\pi)\,$ holds. Now for each level the value of $C_{k+1,\alpha}$ would be computed as
\begin{equation}\label{eq:C_values}
    C^{2}_{k+1,\alpha}=\frac{f_{k+1,\alpha}(\,\pi)\,}{L_{k+1,\alpha}(\,\pi)\,}=\frac{2^{\alpha+1}}{L_{k+1,\alpha}(\,\pi)\,},\qquad k=0,1,2,\dots \ell-1,
\end{equation}
where
\[L_{k+1,\alpha}(\,x)\,=\frac{1}{2}\Big[\,f_{k,\alpha}\Big(\,\frac{x}{2}\Big)\,p^{2}\Big(\,\frac{x}{2}\Big)\,+f_{k,\alpha}\Big(\,\pi-\frac{x}{2}\Big)\,p^{2}\Big(\,\pi-\frac{x}{2}\Big)\,\Big]\,.\]

Since
$L_{k+1,\alpha}(\,\pi)\,=\,f_{k,\alpha}\left(\,\frac{\pi}{2}\right)\,$ converges to 4 as $k$ diverges, then we have
\begin{equation}\label{eq:limck}
\lim_{k\to\infty} C_{k,\alpha} = 2^\frac{\alpha-1}{2}, \qquad \text{ for } \alpha\in(1,\,2],
\end{equation}
as can be seen in \figurename~\ref{fig:Ck}.
Moreover, the convergence of $C_{k,\alpha}$ is quite fast, especially for $\alpha$ close to 2.

\begin{figure}
		\centering
			\includegraphics[width=0.6\textwidth]{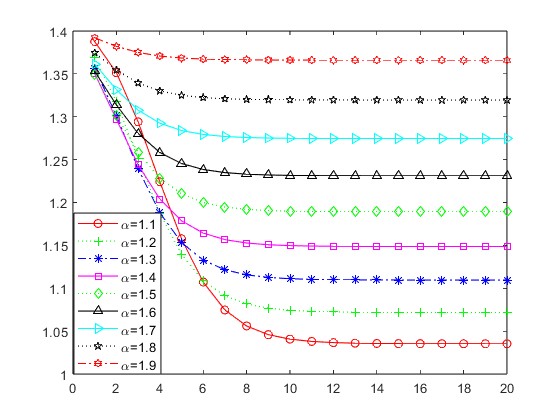}
		\caption{Plot of $C_{k,\alpha}$ vs the level $k$, for different values of $\alpha$.}
	\label{fig:Ck}
\end{figure}

We observe that the sequence of the coarser symbols satisfies
\[
\frac{ f_{\ell,\alpha}(x)}{c(\alpha)}\,\geq \frac{f_{\ell-1,\alpha}(x)}{c(\alpha)} \geq \dots \geq \frac{f_{0,\alpha}(x)}{c(\alpha)} \geq l(x),
\]

for any $\alpha \in (1,2)$, as can be seen in \figurename~\ref{fig:fk}.
Combining this fact with
\[
\|f_{k,\alpha}\|_\infty = f_{k,\alpha}(\pi)
= f_{0,\alpha}(\pi)=c(\alpha)2^{\alpha+1}, \qquad k=1\dots,\ell
\]
for any $\alpha \in (1,2)$, we obtain that the ill-conditioning of the coarser linear systems decreases moving on coarser grids.

\begin{figure}
		\centering			
  \includegraphics[width=0.6\textwidth]{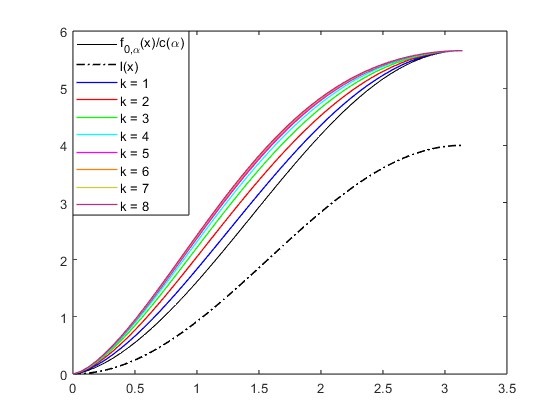}
		\caption{Plots of $f_{k+1,\alpha}/c(\alpha)$ for $\alpha=1.5$ and $k=0,\dots,8$, with $C_{k,\alpha}$ computed according to \eqref{eq:C_values}.}
	\label{fig:fk}
\end{figure}

\subsection{Smoothing Property}\label{subsec:smooth_property}
If weighted Jacobi is used as smoother, then the smoothing property \eqref{eq:Smoothing_pro} is satisfied whenever the smoother converges \cite{chan1998multigrid,ruge1987algebraic}. Since the matrix
$A^{\alpha}_{M_k}$ is symmetric positive definite, the weighted Jacobi method
$$S_k=I_{M_k}-\omega_k D_k^{-1}A^{\alpha}_{M_k}$$
is well-defined, with $D_k$ being the diagonal of ${A^{\alpha}_{M_k}}$. Moreover, it convergences for $ 0<\omega_k<{2}/{\rho(D_k^{-1}{A^{\alpha}_{M_k}})}$,  where $\rho(D_k^{-1}{A^{\alpha}_{M_k}})$ is the spectral radius of $D_k^{-1}{A^{\alpha}_{M_k}}$.

When  ${A^{\alpha}_{M_k}}=T_{M_k}(f_{0,\alpha})$, as for the geometric approach, then $D_k={a}^{(0)}_{0}I_{M_k}$, where ${a}^{(0)}_{0}$ is the Fourier coefficient of order zero of $f_{0,\alpha}$, and thus it holds
\begin{equation}\label{eq:Sk}
S_k=I_{M_k}-\frac{\omega_k}{2\alpha}A^{\alpha}_{M_k},
\end{equation}
since ${a}^{(0)}_{0}=2c(\alpha)\alpha$.

\begin{lemma}\label{Lemma2_Smoothing}
Let $A^\alpha_{M_k}=T_{M_k}(f_{0,\alpha})$ with $f_{0,\alpha}=f_\alpha$ defined as in \eqref{eq:symbol_ftn} and let $S_{M_k}$ defined as in \eqref{eq:Sk}. If we choose
\begin{equation}\label{eq:omk}
  0<\omega_k<\frac{\alpha}{2^{\alpha-1}},
\end{equation}
 then there exist $\sigma_{k}$ such that the smoothing property in \eqref{eq:Smoothing_pro} holds true.

Moreover, choosing
\begin{equation}\label{eq:omegak}
\omega_k=\frac{2^{2-\alpha}\alpha}{3},
\end{equation}
it holds
\begin{equation}\label{sigmak}
 \sigma_k \geq \sigma = \frac{2^{-\alpha}}{9}>0.
\end{equation}
\end{lemma}
\begin{proof}
    Since $f_{0,\alpha}$ is nonnegative, from Lemma 4.2 in \cite{moghaderi2017spectral} it follows
  $$ 0<\omega_k<\frac{2{a}^{(0)}_{0}}{\|f_{0,\alpha}\|_{\infty}}=\frac{\alpha}{2^{\alpha-1}},
  $$
since $\|f_{0,\alpha}\|_\infty=f_{0,\alpha}(\pi)=c(\alpha){2^{\alpha+1}}$.
  Moreover, for $\omega_k$ chosen as in \eqref{eq:omegak}, the following condition
$$\left(I_{M_k}-\frac{\omega_k}{2\alpha}A^{\alpha}_{M_k}\right)^2 \leq
I_{M_k}-\frac{\sigma_k}{2\alpha}A^{\alpha}_{M_k}$$
 in the proof of Lemma 4.2 in \cite{moghaderi2017spectral} is satisfied for all $\sigma_k \geq \sigma$ in \eqref{sigmak}.
\end{proof}

From the previous lemma, following the same idea as given in \cite{tempered}, we propose to use the following Jacobi parameter
\begin{equation}\label{eq:omega}
\omega_k=\omega^\star=\frac{2^{2-\alpha}\alpha}{3}, \qquad \forall \, k\geq0,
\end{equation}
which provides a good convergence rate as confirmed in the numerical results in Section~\ref{sec:num}.

%---------------------------------------------------------------------------------
\subsection{Approximation property and level independency}
According to Remark \ref{rem:gamma}, the uniform boundness of $\delta_k$ is ensured by proving equation \eqref{eq:levindToep}.

First of all, we prove that
for $k\in\mathbb{N}_{0}$, the sequence of functions $g_{k,\alpha}(x)=\frac{p_{k,\alpha}^{2}(\pi-x)}{f_{k,\alpha}(x)}$ is monotonic increasing in $x$.
Recalling the expression of $p_{k,\alpha}$ in \eqref{eq:symbpk}, differentiating $g_{k,\alpha}$ we have
\begin{align*}
    \frac{d}{dx}{g_{k,\alpha}(x)}=\frac{{C_{k+1,\alpha}^{2}}(1-\cos{x})}{f_{k,\alpha}^{2}(x)}\big( 2\sin{x}{f_{k,\alpha}(x)}-(1-\cos{x})\frac{d}{dx}{f_{k,\alpha}(x)}\big),
\end{align*}
and hence it is equivalent to proving that
$$\mathbf{g}_{k,\alpha}(x)=2\sin{x}{f_{k,\alpha}(x)}-(1-\cos{x})\frac{d}{dx}{f_{k,\alpha}(x)}$$
is nonnegative for all $k\in \mathbb{N}_{0}$. This result can be proved by induction on $k$.

For $k=0$, replacing the expression of $f_{0,\alpha}$ (see equation \eqref{eq:symbol_ftn}) in
$\mathbf{g}_{k,\alpha}$, by direct computations we obtain $\mathbf{g}_{k,\alpha}(x) \geq 0$,
$\forall \, x \in [0, \pi]$.

Assuming that for $k=n$ it holds $\mathbf{g}_{n,\alpha}(x) \geq 0$,
$\forall \, x \in [0, \pi]$, we have to prove that the same is true for $k=n+1$.

Applying the recurrence \eqref{eq:f_at_level} to $f_{n+1,\alpha}$, we have
\[
\mathbf{g}_{n+1,\alpha}(x) \,\geq \,
p^{2}\left(\frac{x}{2}\right) S_1(x) - p^{2}\left(\pi-\frac{x}{2}\right) S_2(x),
\]
where $p(x) = 1+\cos(x)$, as defined in \eqref{eq:symbpk}, and
\begin{align*}
S_{1}(x)&=\sin{\Big(\,\frac{x}{2}\Big)}{f_{n,\alpha}\Big(\,\frac{x}{2}\Big)}-\frac{1}{2}\Big(1-\cos{\Big(\,\frac{x}{2}\Big)}\Big)\frac{d}{dx}{f_{n,\alpha}\Big(\,\frac{x}{2}\Big)},\\
S_{2}(x)&=\sin{\Big(\frac{x}{2}\Big)}{f_{n,\alpha}\Big(\pi-\frac{x}{2}\Big)}+\frac{1}{2}\Big(1+\cos{\Big(\,\frac{x}{2}\Big)}\Big)\frac{d}{dx}{f_{n,\alpha}\Big(\,\pi-\frac{x}{2}\Big)}.
\end{align*}
Thanks to the inductive assumption and direct computation, we obtain the desired result $\mathbf{g}_{n+1,\alpha}(x) \geq 0$, $\forall \, x \in [0, \pi]$, i.e.,
\begin{align*}
    \frac{d}{dx}{g_{k,\alpha}(x)}\geq0,\quad \forall k\in \mathbb{N}_{0}.
\end{align*}

%In Figure \ref{fig: Plots of g_k_approx_property}, where the functions $\mathbf{g}_{k,\alpha}$
%are depicted for different values of $k$ and $\alpha$, we note that varying $k$ it holds also
%\begin{equation}\label{eq:grow}
%\mathbf{g}_{k+1,\alpha}(x) \geq \mathbf{g}_{k,\alpha}(x)
%\end{equation}
Figure \ref{fig: Plots of g_k_approx_property} depicts the function $\mathbf{g}_{k,\alpha}$
for different values of $k$ and $\alpha$.

 \begin{figure}
		\centering
		%	\begin{center}$
		\begin{subfloat}[$\alpha=1.2$]
			{\resizebox*{4.7cm}{!}{\includegraphics[width=\textwidth]{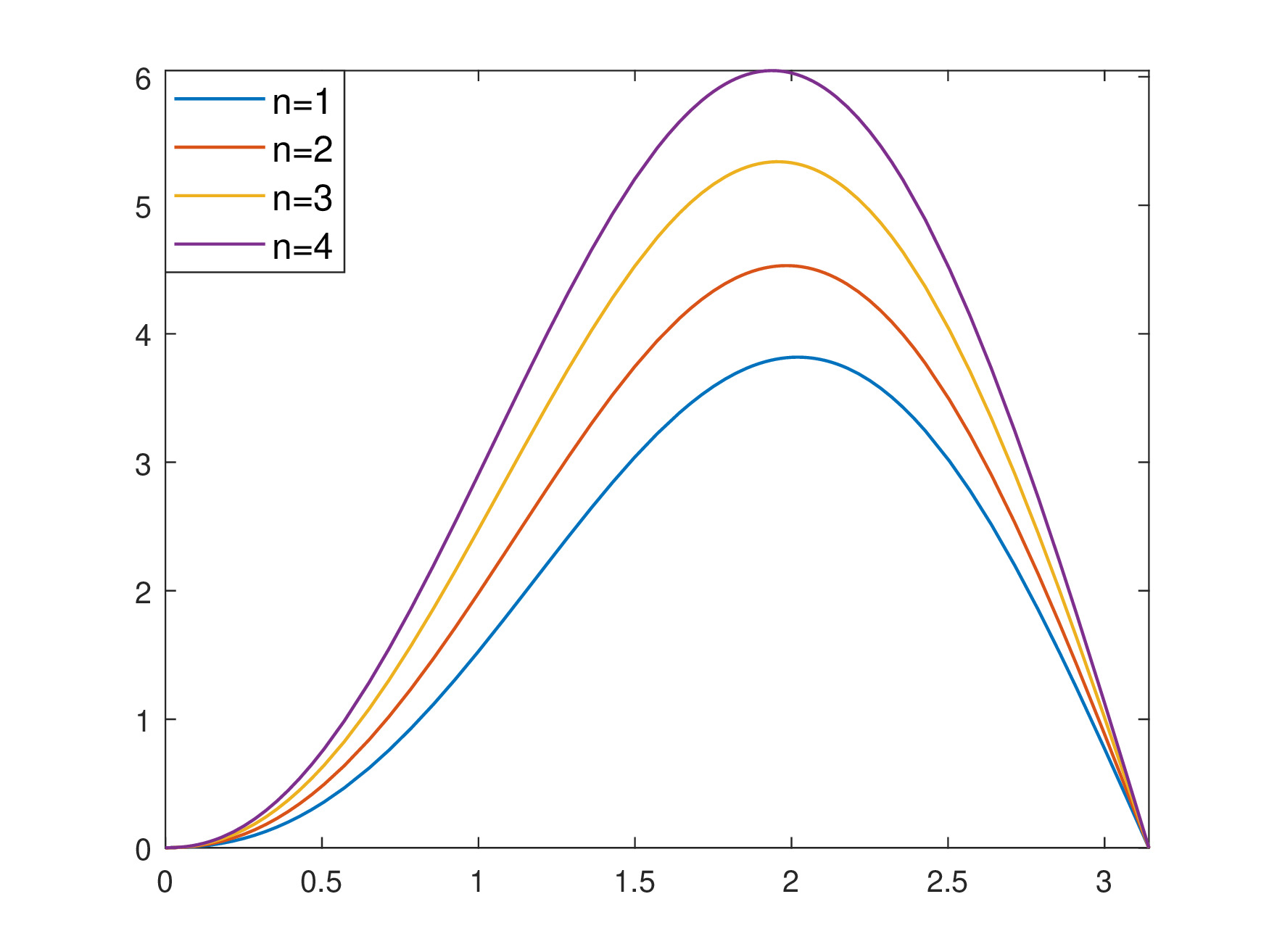}}}
		\end{subfloat}
		\begin{subfloat}[$\alpha=1.5$]
			{\resizebox*{4.7cm}{!}{\includegraphics[width=\textwidth]{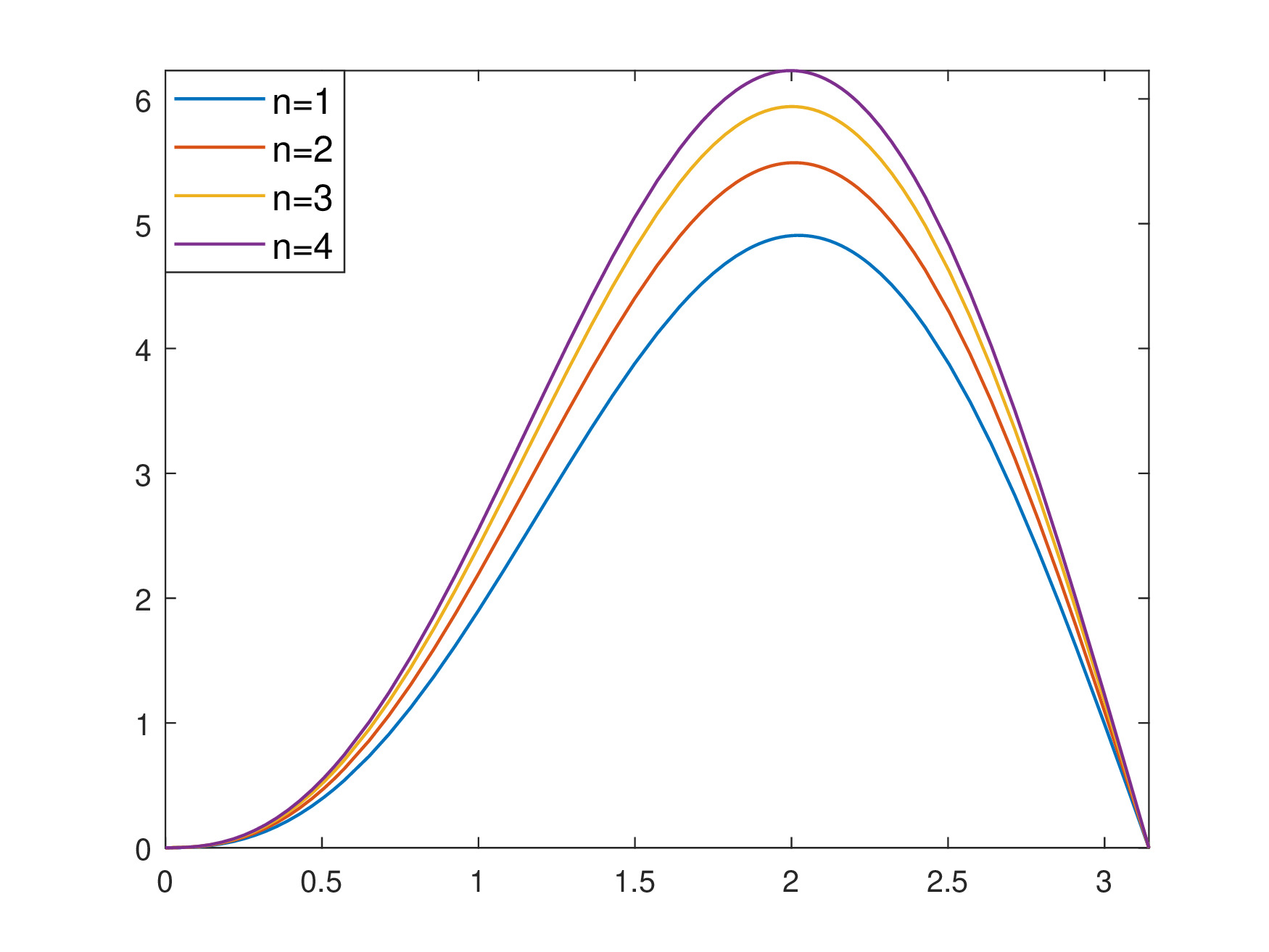}}}
		\end{subfloat}
		\begin{subfloat}[$\alpha=1.8$]
			{\resizebox*{4.7cm}{!}{\includegraphics[width=\textwidth]{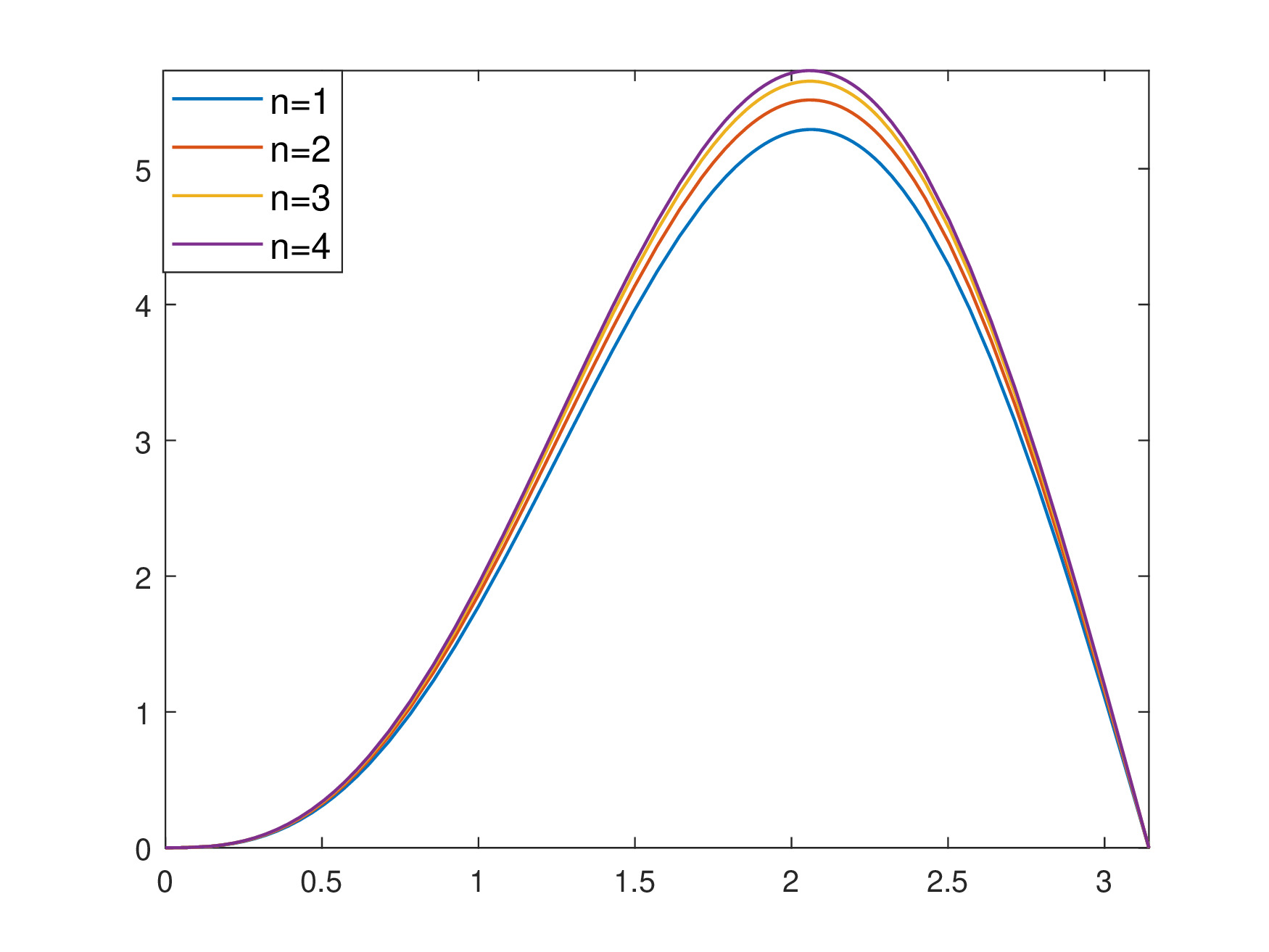}}}
		\end{subfloat}
		\caption{Plot of  $\mathbf{g}_{n,\alpha}(x)$ for different values of $\alpha$.}
	\label{fig: Plots of g_k_approx_property}
 \end{figure}

In conclusion, by the monotonicity of $g_{k,\alpha}(x)$, for all $k\geq 0$,  we have
\[\smash{\displaystyle\max_{x \in[0,\pi]}} g_{k,\alpha}(x)=g_{k,\alpha}(\pi)=\frac{p_{k,\alpha}^{2}(0)}{f_{k,\alpha}(\pi)} = \frac{4C^{2}_{k,\alpha}}{c(\alpha)2^{\alpha+1}},\]
and since the sequence $C_{k,\alpha}$ is bounded thanks to equation \eqref{eq:limck} (see also Figure \ref{fig:Ck}), we have that $\{\gamma_k\}_k$, and hence $\{\delta_k\}_k$, is bounded.

Combining the previous result on $\delta_k$ with equation \eqref{sigmak} for the smoothing analysis, it holds \eqref{eq:levind}, i.e., the level independency is satisfied.

%\begin{itemize}
%    \item  If $p_{k,\alpha}(x)=C_{k+1,\alpha}\,p(x)$, then at each level $f_{k+1,\alpha}(\pi)=f_{k,\alpha}(\pi)$. Then $C_{k+1,\alpha}\Big\|\frac{p(\pi-x)^{2}}{f_{k,\alpha}}\Big\|_{\infty}$ is decreasing at each consecutive level because the value of $C_{k+1,\alpha}$ is decreasing (See \figurename~\ref{fig: C_Values_at_each_level}).
%    \item  If $p_{k,\alpha}(x)=\sqrt{2}\,p(x)$, then at each level $f_{k+1,\alpha}(\pi)>f_{k,\alpha}(\pi)$. Thus  ${\displaystyle\max} \{g_{k,\alpha}(\pi)\}=g_{0,\alpha}(\pi),\quad\forall k$.
%\end{itemize}

%--------------------------------------------------------------------------------
\section{Band Approximation}\label{sec:band_approximation}

%Multigrid methods are often applied as preconditioners for Krylov methods, like conjugate gradient (CG) and GMRES, instead of as solvers.  In such case, to an approximation of $A^{\alpha}_{M}$ instead of at the original coefficient matrix. In this section, we study a band approximation of $A^{\alpha}_{M}$ to apply our multigrid method. The advantage of using a band matrix is the possibility of applying the Galerkin approach with a linear cost in $M$ instead of $O(M\log(M))$ and inheriting the V-cycle optimality from the results in \cite{ADS}.

Multigrid methods are often applied as preconditioners for Krylov methods, and often to an approximation of $A^{\alpha}_{M}$ instead of at the original coefficient matrix. On the same line as what has been done in \cite{donatelli2020multigrid}, in this section, we discuss a band approximation of $A^{\alpha}_{M}$ to combine with the multigrid method discussed in the previous section. The advantage of using a band matrix is in the possibility of applying the Galerkin approach with a linear cost in $M$ instead of $O(M\log M)$ and inheriting the V-cycle optimality from the results in~\cite{ADS}.

Starting from the truncated Fourier sum of order $s$ of the symbol $f_{\alpha}$
  \begin{equation}\label{eq:band_app}
        g_s(\,x)\,=2\sum_{k=0}^{s}{g_{k}^{(\alpha)}\cos{\left((k-1)x\right)}},
\end{equation}
 we consider as band approximation of $A^{\alpha}_{M}$ the associated Toeplitz matrix $ {_{s}\mathcal{\bf B}^{\alpha}_{M}}=T_M(g_s)$ explicitly given by
  \begin{equation}\label{eq:Band_matrix}
    {_{s}\mathcal{\bf B}^{\alpha}_{M}}=-
	\begin{bmatrix}
		& 2{g_{1}^{(\alpha)\,}} & {g_{0}^{(\alpha)\,}}+{g_{2}^{(\alpha)\,}} & \cdots & {g_{s}^{(\alpha)\,}}&  &  &\\
		& {g_{0}^{(\alpha)\,}}+{g_{2}^{(\alpha)\,}} & 2{g_{1}^{(\alpha)\,}} & {g_{0}^{(\alpha)\,}}+{g_{2}^{(\alpha)\,}} &\cdots &{g_{s}^{(\alpha)\,}} & &\\
		& \vdots & \ddots & \ddots & \ddots &\ddots &{g_{s}^{(\alpha)\,}} &\\
		& {g_{s}^{(\alpha)\,}} & \ddots & \ddots &\ddots &\ddots &\vdots &\\
		&  & {g_{s}^{(\alpha)\,}}& \ddots & {g_{0}^{(\alpha)\,}}+{g_{2}^{(\alpha)\,}} & 2{g_{1}^{(\alpha)\,}} & {g_{0}^{(\alpha)\,}}+{g_{2}^{(\alpha)\,}} &\\
		& &  &{g_{s}^{(\alpha)\,}}& \cdots & {g_{0}^{(\alpha)\,}}+{g_{2}^{(\alpha)\,}} & 2{g_{1}^{(\alpha)\,}}&
	\end{bmatrix}_{{M\times M}}.
\end{equation}
The emerging banded matrix is defined as
\begin{align}\label{eq:band_coef}
    {_{s}\Tilde{A}^{\alpha}_{M}}=\overline{c}\,\,_{s}\mathcal{\bf B}_{M}^{\alpha},
\end{align}
%with $_{s}\mathcal{\bf B}_{M}^{\alpha}={_{s}\mathcal{G}^{\alpha}_{M}}+{_{s}{\mathcal{G}^{\alpha}_{M}}^{T}}$.

% In the numerical results, we choose $s$ such that
% \begin{equation}\label{eq:s}
% M=2^{s}-1.
% \end{equation}

As already discussed in \cite{moghaderi2017spectral,donatelli2020multigrid}, concerning the application of our multigrid method to a linear system with the coefficient matrix $_{s}\Tilde{A}^{\alpha}_{M}$, we have the following feature thanks to the results in \cite{ADS,arico2007}:
\begin{itemize}
\item The computation of the coarser matrices by the Galerkin approach \eqref{eq:gal}, using the projector defined in \eqref{eq:Pk}, preserves the band structure at the coarser levels, see \cite[Proposition 2]{arico2007}.
\item If the function $g_s$ is nonnegative, then the V-cycle is optimal, i.e., has a constant convergence rate independent of the size $M$ and a computational cost per iteration proportional to $O(sM)$, that is $O(M\log M)$ with the choice of $s$ such that $M=2^s-1$.
\end{itemize}

About the spectral features of preconditioned Toeplitz matrices with Toeplitz preconditioners, we recall the well-known results of both localization and distribution type in the sense of Weyl in the following theorem (see \cite{Sergo, GLT_I_book}).
%First we report the definition of the spectral distribution in the sense of the eigenvalues.

% \begin{definition}\label{def-distribution}
% 	Let $\kappa:[a,b]\to\mathbb{C}$ be a measurable function, defined on $[a,b]\subset\mathbb R$, let  $\mathcal C_0(\mathbb C)$ be the set of continuous functions with compact support over
% 	$\mathbb C$ and let $\{\mathcal{A}_n\}_n$ be a sequence of matrices of size $n$ with eigenvalues $\lambda_j(\mathcal{A}_n)$, $j=1,\ldots,n$. We say that $\{\mathcal{A}_n\}_n$ is distributed as the pair $(\kappa,[a,b])$ in the sense of the eigenvalues, and we write $$\{\mathcal{A}_n\}_n\sim_\lambda(\kappa,[a,b]),$$ if the following limit relation holds for all $F\in\mathcal C_0(\mathbb C)$:
% 	\begin{align}\label{distribution:sv-eig}
% 	\lim_{n\to\infty}\frac{1}{n}\sum_{j=1}^{n}F(\lambda_j(\mathcal{A}_n))=
% 	\frac1{b-a}\int_a^b F(\kappa(t)) dt.
% 	\end{align}
% \end{definition}
% %\begin{remark} \label{rem:comp}
% 	When $\kappa$ is continuous, an informal  interpretation of the limit relation \eqref{distribution:sv-eig} is that when the matrix-size is sufficiently large, the eigenvalues of $\mathcal{A}_n$ can be approximated by a sampling of $\kappa$ on a uniform equispaced grid of the interval $[a,b]$.
% %\end{remark}

%The following theorem is a collection of relevant results available in the quoted literature.

\begin{theorem}\label{th-prec}
Let $f$ be real-valued, $2\pi$-periodic, and continuous and  $g$ be nonnegative, $2\pi$-periodic, continuous, with $\max g>0$. Then
\begin{itemize}
\item $T_M(g)$ is positive definite for every $M$;
\item if $r=\inf \frac{f}{g}$, $R=\sup \frac{f}{g}$ with $r<R$, then all the eigenvalues of $T_M^{-1}(g)T_M(f)$ lie in the open set $(r,R)$ for every $M$; %(the case where $r=R$ is trivial since $P_n(f,g)= rI_n$, $I_n$ being the identity of size $n$);
\item $\{T_M(f)\}_M,\{T_M(g)\}_M,\{T_M^{-1}(g)T_M(f)\}_M$ are distributed in the Weyl sense as $f,g,\frac{f}{g}$, respectively, i.e. for all $F$ in the set of continuous functions with compact support over $\mathbb C$ it holds
	\begin{align}\label{distribution:sv-eig}
	\lim_{n\to\infty}\frac{1}{M}\sum_{j=1}^{M}F(\lambda_j(\mathcal{A}_M))=
	\frac1{b-a}\int_a^b F(\kappa(t)) dt.
	\end{align}
with $A_M=T_M(f),T_M(g),T_M^{-1}(g)T_M(f)$, $\kappa=f,g,\frac{f}{g}$, $\lambda_j(\mathcal{A}_M)$ the eigenvalues of $A_M$, and $[a,b]=[-\pi,\pi]$.
\end{itemize}
\end{theorem}

%\begin{remark} \label{rem:th-prec}
%	Notice that the above theorem is a simplification of what is available in the literature (see \cite{Sergo, GLT_I_book}), where $f,g$ are just Lebesgue integrable as in Definition \ref{def1_fourier_coefficients}, $\kappa=\frac{f}{g}$ measurable, $M_g, r, R$ are the essential supremum of $g$, the essential infimum of $\kappa$, the essential supremum of $\kappa$, respectively.

%Here we reported a simplified version tailored for our problem setting. In this respect, since all the involved Toeplitz matrices are real symmetric, the related generating functions $f=f_{0,\alpha}$ and $g_{s}=\mathcal{B}^{s}_{0,\alpha}$ are necessarily even: as a consequence the interval $[a,b]$ in Definition \ref{def-distribution} can be defined as $[0,\pi]$ instead of the $[-\pi,\pi]$, as customary in the Toeplitz setting.  	
%\end{remark}

%As already recalled in the previous remark, by referring to Theorem \ref{th-prec}, $g_{s}=\mathcal{B}^{s}_{0,\alpha}$, $f=f_{0,\alpha}$, and the important quantities are the minimum and the maximum of $f,\,g_{s},\,\kappa=\kappa_s$.

%In other words, since the symmetrized preconditioner has to approximate the identity as much as possible, the approximation $\approx$ indicated in relation (\ref{eq:band_app}) has to be intended in a relative sense.

In other words, a good Toeplitz preconditioner for a Toeplitz matrix generated by $f$ should be generated by $g$ such that $\frac{f}{g}$ is close to $1$ in a certain metric, for instance in $L^\infty$ or in $L^1$ norm.

About the band preconditioner in \eqref{eq:Band_matrix}, Figure \ref{fig: f_{g_s}} shows the graphs of $g=g_s$ and  $f=f_{\alpha}$ for two different values of $\alpha$, while $\delta_s=f_{\alpha}-g_s$ are depicted in Figure \ref{fig: delta_ftn}, for some values of $s$. Obviously, the quality of the approximation grows increasing $s$, but it is good enough already for a small $s$, in particular when $\alpha$ approaches the value $2$ (see the scale of the $y-$axis in Figure \ref{fig: delta_ftn}).

More importantly, in the light of Theorem \ref{th-prec}, Figure \ref{fig: kappa_ftn} depicts the functions $\kappa_s-1$ with $\kappa_s=\frac{f_{\alpha}}{g_s}$. This quantity measures the relative approximation, which is the one determining the quality of the preconditioning since $(\kappa_s-1,[0,\pi])$ is the distribution function of the shifted preconditioned matrix-sequence
\[
\{T_M^{-1}(g)T_M(f)-I_M\}_M, \ \ \ \ g=g_s,\ f=f_{\alpha},
\]
in the sense of Weyl and in accordance again with Theorem \ref{th-prec}.
Note that the function $\kappa_s-1$ is almost zero except around the zero of $f_{\alpha}$ at the origin.

\begin{figure}
		\centering
		%\begin{center}
		\begin{subfloat}[$\alpha=1.2$]
			{\resizebox*{6cm}{!}{\includegraphics[width=\textwidth]{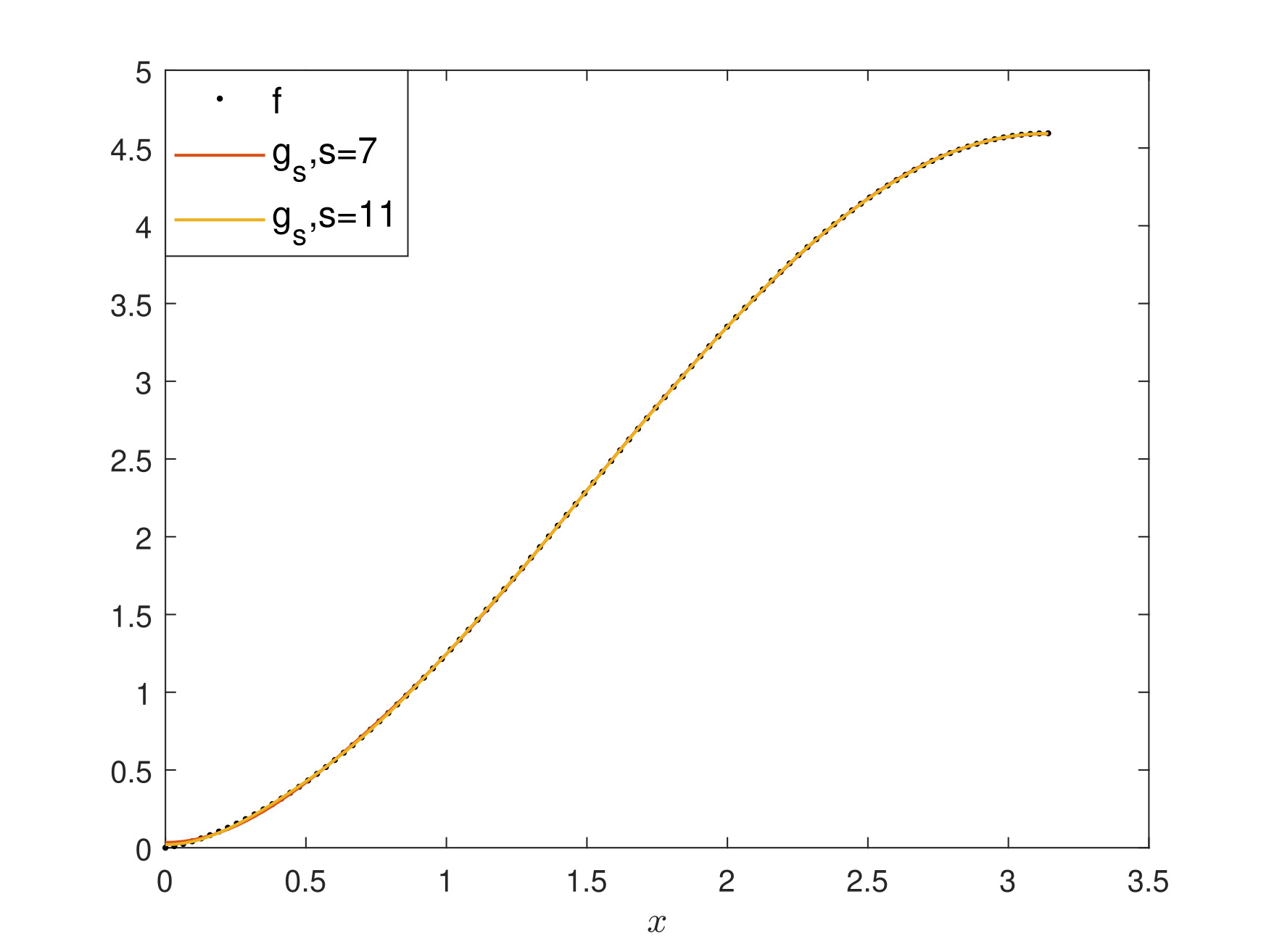}}}
		\end{subfloat}
  \begin{subfloat}[$\alpha=1.8$]
			{\resizebox*{6cm}{!}{\includegraphics[width=\textwidth]{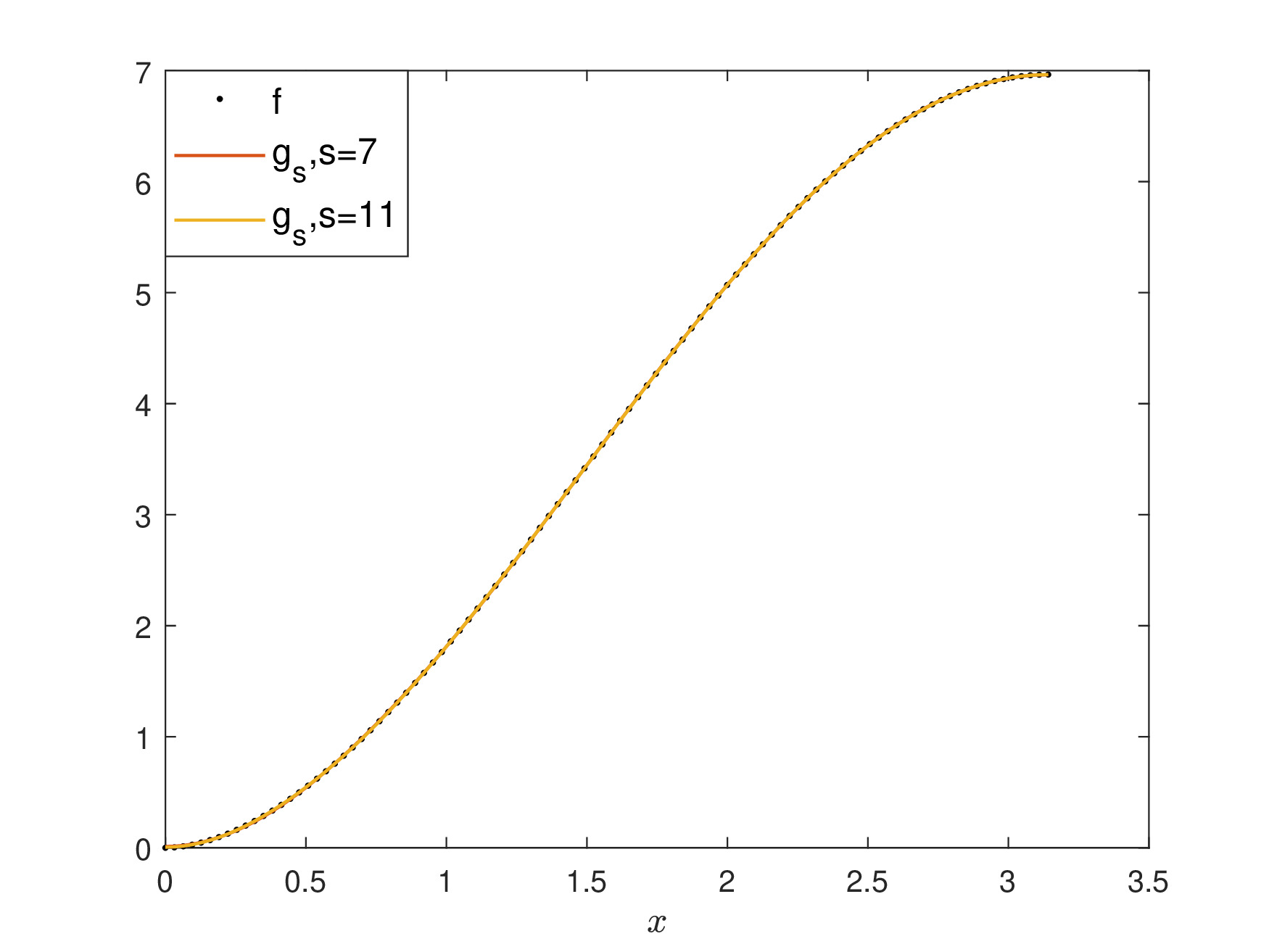}}}
		\end{subfloat}
		\caption{Plots of $f$ and ${g}_{s}$.}
	\label{fig: f_{g_s}}
	\end{figure}
 %%%%%%
 \begin{figure}
		\centering
		%\begin{center}
		\begin{subfloat}[$\alpha=1.2$]
			{\resizebox*{6cm}{!}{\includegraphics[width=\textwidth]{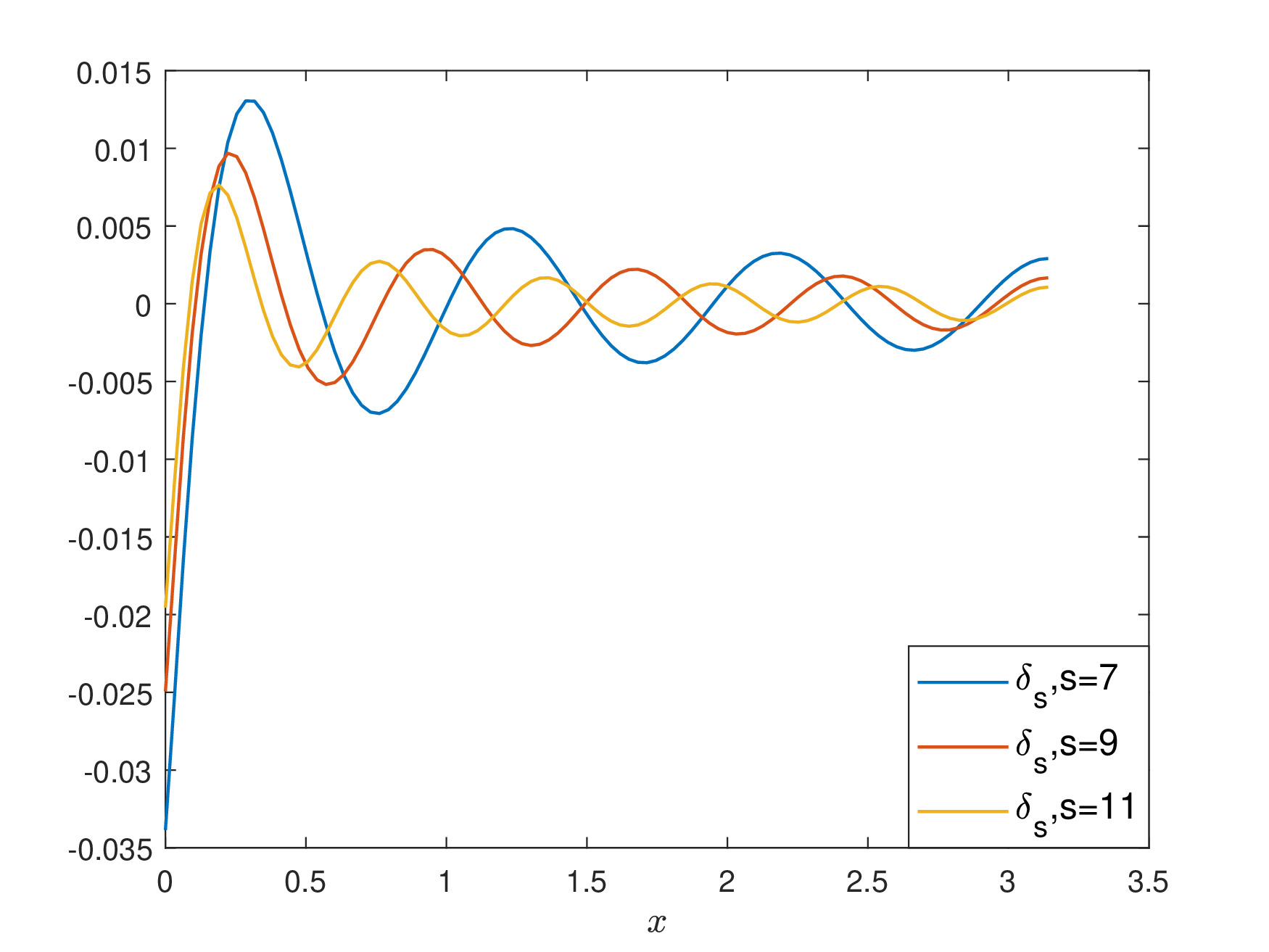}}}
		\end{subfloat}
  \begin{subfloat}[$\alpha=1.8$]
			{\resizebox*{6cm}{!}{\includegraphics[width=\textwidth]{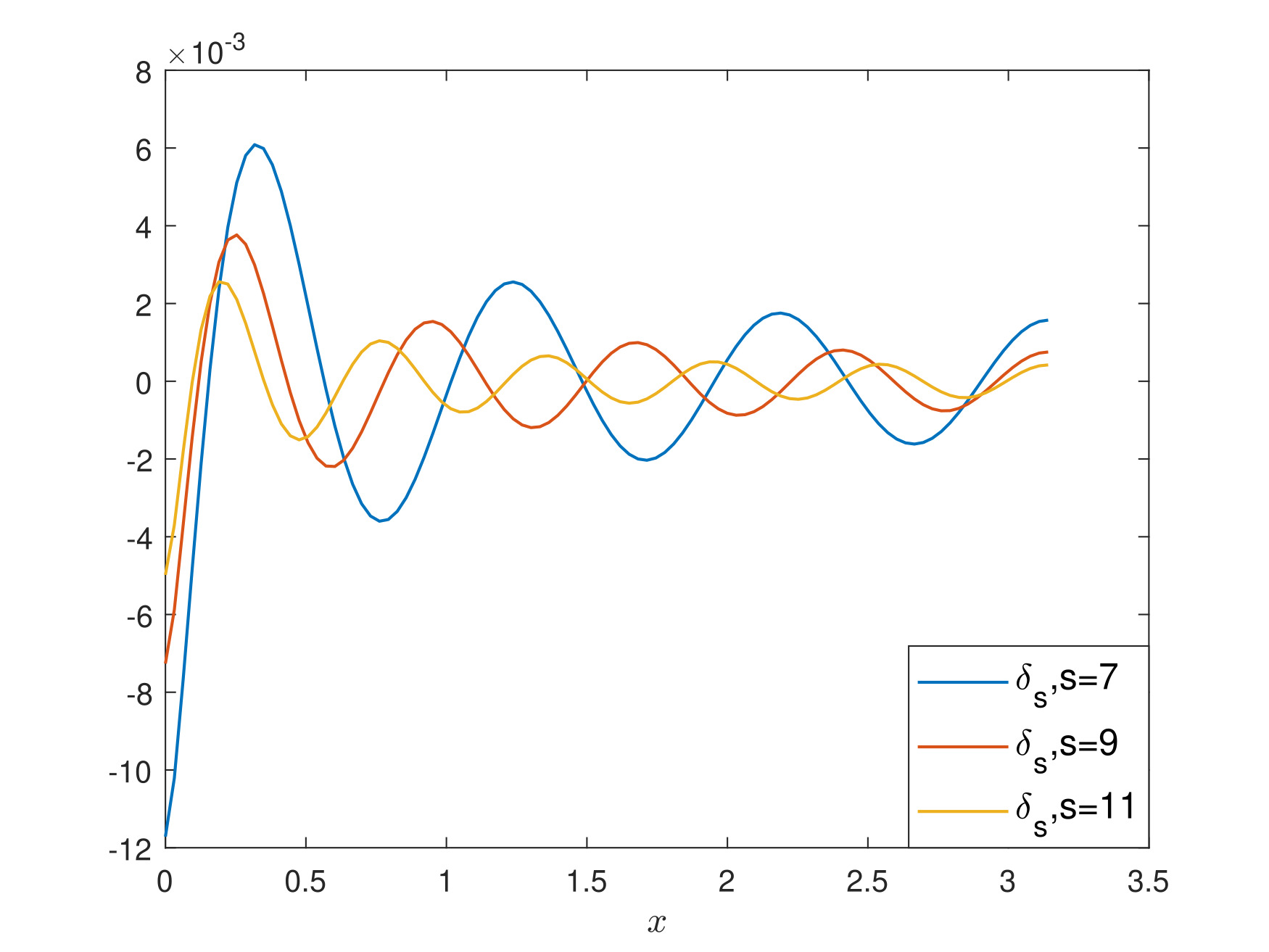}}}
		\end{subfloat}
		\caption{Plots of ${\delta}_{s}$.}
	\label{fig: delta_ftn}
	\end{figure}
 %%%%%%%
  \begin{figure}
		\centering
		%\begin{center}
		\begin{subfloat}[$\alpha=1.2$]
			{\resizebox*{6cm}{!}{\includegraphics[width=\textwidth]{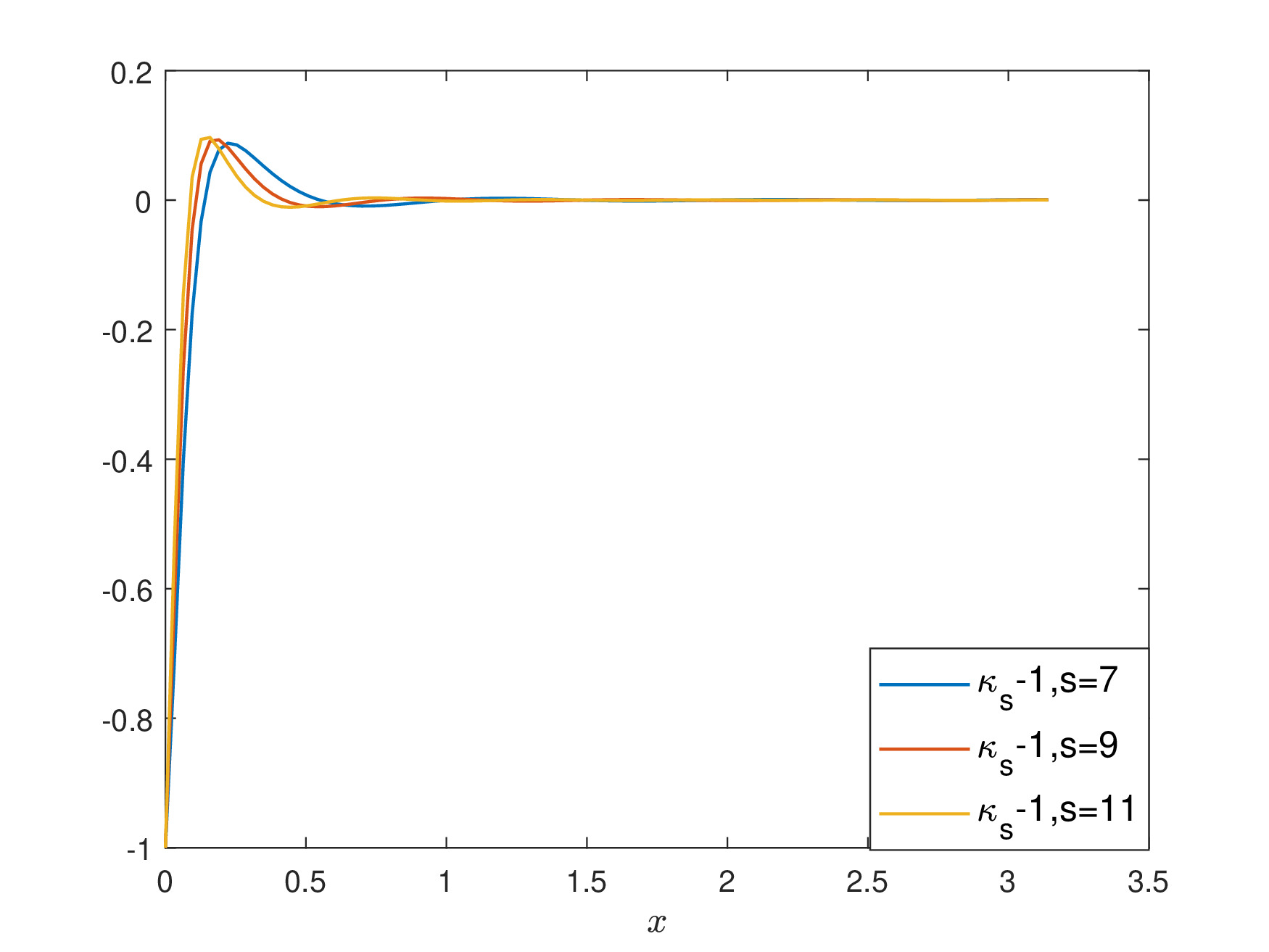}}}
		\end{subfloat}
  \begin{subfloat}[$\alpha=1.8$]
			{\resizebox*{6cm}{!}{\includegraphics[width=\textwidth]{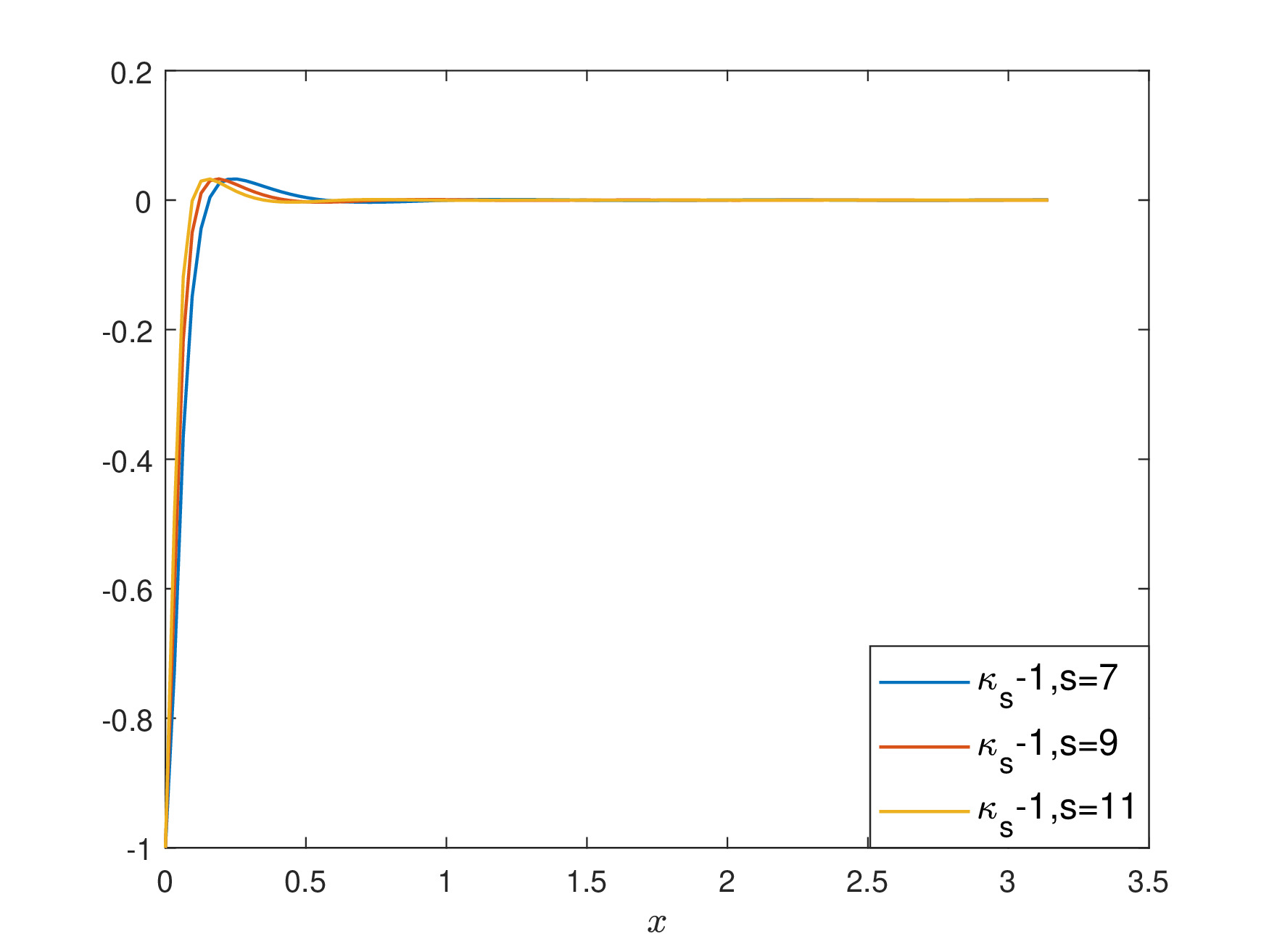}}}
		\end{subfloat}
		\caption{Plots of ${\kappa}_{s}-1$.}
	\label{fig: kappa_ftn}
	\end{figure}

%-------------------------------------------
\section{Numerical Examples}\label{sec:num}
In this section, we present some numerical examples, to verify the effectiveness of the MGM used both as a solver and preconditioner.
In what follows:
\begin{itemize}
    \item $V_{\nu_1}^{\nu_2}$ consists of a Galerkin V-cycle with linear interpolation as grid transfer operator and $\nu_1$ and $\nu_2$ iterations of pre and post-smoother weighted Jacobi, respectively;
    \item $P_{C}$ and $P_{S}$ represent the Chan \cite{chan1988optimal} and Strang circulant preconditioner, respectively;
    \item $P{_{s}V_{\nu_1}^{\nu_2}}$ denotes the banded preconditioner defined as in %Section \ref{sec:band_approximation}
    equation \eqref{eq:Band_matrix} and inverted using MGM with Galerkin approach $V_{\nu_1}^{\nu_2}$;
    \item $PV_{\nu_1}^{\nu_2}$ and $\widetilde{P}V_{\nu_1}^{\nu_2}$ are multigrid preconditioners, both as Galerkin and Geometric approach, respectively.
\end{itemize}
The aforementioned preconditioners are combined with CG and GMRES computationally performed using built-in \textit{pcg} and \textit{gmres} Matlab functions.
The stopping criterion is chosen as
\begin{equation*}
    \frac{\| r^{k}\|}{\| r^{0}\|}<10^{-8},
\end{equation*}
where $\|\cdot\|$ denotes the Euclidean norm, $r^{k}$ is the residual vector at the $k$-th iteration. The initial guess is fixed as the zero vector. All the numerical experiments are run by using MATLAB R2021a on HP 17-cp0000nl computer with configuration, AMD Ryzen 7 5700U with Radeon Graphics CPU and 16GB RAM.

\paragraph{Example 1.} Consider the following 1D-RFDE given in equation \eqref{eq:RFDE}, with $\Omega=[\,0,1\,]$ and the source term built from the exact solution
\begin{equation}\label{eq:Ex1}
    u(x)=x^{2}(\,1-x\,)^{2}.
\end{equation}

\begin{table}
	\caption{Example 1: Number of iterations for $\alpha=1.2,\,1.5$, and $1.8$ with $\omega^\star=0.6964,\,0.7071$, and $0.6892$, respectively.}
	\begin{small}
		\setlength{\tabcolsep}{5.0pt}
		\begin{center}
			\begin{tabular}{c c c c c c c c c c c c c c c c c c c c c}
		%	\hline
				\hline
				\noalign{\vskip 1mm}
				\multirow{1}{*} {}&
				\multicolumn{3}{c}{${}$} &
				\multicolumn{3}{c}{${\mbox{\bf Galerkin}}$} &
				\multicolumn{2}{c}{${}$} &
				\multicolumn{3}{c}{${\mbox{\bf Geometric}}$} &
                \multicolumn{3}{c}{${}$} \\
               \cmidrule(r){4-9}\cmidrule(r){10-15}
				%%%%%%%%%%%%%%%%%
				\noalign{\vskip 1mm}
				\multirow{1}{*} {}&
				\multicolumn{2}{c}{${}$} &
				\multicolumn{2}{c}{${\mbox{TGM}}$} &
				\multicolumn{1}{c}{${}$}&
				\multicolumn{3}{c}{${\mbox{V-Cycle}}$} &
				\multicolumn{3}{c}{${\mbox{TGM}}$} &
				\multicolumn{3}{c}{${\mbox{V-Cycle}}$} &
                \multicolumn{6}{c}{${\mbox{\bf Preconditioners}}$}\\
				\cmidrule(r){4-6}\cmidrule(r){7-9}\cmidrule(r){10-12}\cmidrule(r){13-15}\cmidrule(r){16-21}
				$\alpha$&$M+1$& ${\mbox{CG}}$& ${V_{0}^{1}}$&${V_{1}^{0}}$& ${V_{1}^{1}}$ &${V_{0}^{1}}$&${V_{1}^{0}}$ &${V_{1}^{1}}$& ${V_{0}^{1}}$&${V_{1}^{0}}$& ${V_{1}^{1}}$ &${V_{0}^{1}}$&${V_{1}^{0}}$ &${V_{1}^{1}}$&${s}$& ${P{_{s}{V}_{1}^{1}}}$&${P}{V_{1}^{1}}$& $\widetilde{P}{V_{1}^{1}}$&${P_{C}}$&${P_{S}}$\\ \hline
				%%%%%%%%%%%%%%%%%%%%%%%%%%%%%%%
	${}$&$2^6$                &${32 }$&${17}$&${17}$&${9}$&${17}$&${17}$&${9 }$&${16}$&${16}$&${13}$&${37}$&${38}$&${31}$&${7 }$&${8}$&${6}$&${11}$&${9 }$&${5}$\\
	${}$&$2^7$                &${63 }$&${16}$&${16}$&${9}$&${16}$&${17}$&${10}$&${16}$&${16}$&${13}$&${43}$&${43}$&${34}$&${7 }$&${8}$&${6}$&${12}$&${10}$&${6}$\\
${{\bf{\alpha}}=1.2}$&$2^8$   &${110}$&${16}$&${16}$&${9}$&${16}$&${17}$&${10}$&${16}$&${16}$&${12}$&${48}$&${48}$&${37}$&${9 }$&${10}$&${6}$&${12}$&${12}$&${6}$\\
	${}$&$2^9$                &${178}$&${16}$&${16}$&${9}$&${16}$&${18}$&${10}$&${16}$&${16}$&${12}$&${52}$&${52}$&${40}$&${9 }$&${13}$&${7}$&${12}$&${13}$&${6}$\\
	${}$&$2^{10}$             &${279}$&${15}$&${15}$&${8}$&${16}$&${18}$&${11}$&${15}$&${15}$&${11}$&${55}$&${56}$&${42}$&${11}$&${18}$&${7}$&${12}$&${14}$&${7}$\\
     \hline
    ${}$&$2^{6}$             &${32 }$&${17}$&${17}$&${9}$&${17}$&${17}$&${10}$&${17}$&${17}$&${10}$&${23}$&${22}$&${16}$&${7 }$&${7}$&${6}$&${8}$&${9 }$&${5}$\\
    ${}$&$2^{7}$             &${62 }$&${17}$&${17}$&${9}$&${16}$&${17}$&${9 }$&${17}$&${17}$&${10}$&${25}$&${24}$&${17}$&${7 }$&${8}$&${6}$&${8}$&${11}$&${5}$\\
${{\bf{\alpha}}=1.5}$&$2^{8}$&${111}$&${17}$&${17}$&${9}$&${16}$&${17}$&${10}$&${17}$&${17}$&${10}$&${27}$&${26}$&${19}$&${9 }$&${10}$&${6}$&${8}$&${13}$&${7}$\\
    ${}$&$2^{9}$             &${192}$&${16}$&${16}$&${9}$&${16}$&${18}$&${10}$&${16}$&${16}$&${9 }$&${28}$&${28}$&${20}$&${9 }$&${14}$&${6}$&${9}$&${14}$&${7}$\\
    ${}$&$2^{10}$            &${328}$&${16}$&${16}$&${9}$&${16}$&${18}$&${10}$&${16}$&${16}$&${9 }$&${30}$&${30}$&${20}$&${11}$&${19}$&${7}$&${9}$&${16}$&${8}$\\
    \hline
    ${}$&$2^{6}$             &${32 }$&${17}$&${17}$&${10}$&${17}$&${17}$&${11}$&${17}$&${17}$&${10}$&${18}$&${21}$&${13}$&${7 }$&${8}$&${7}$&${7}$&${10}$&${6}$\\
    ${}$&$2^{7}$             &${64 }$&${17}$&${17}$&${10}$&${17}$&${18}$&${11}$&${17}$&${17}$&${10}$&${19}$&${20}$&${13}$&${7 }$&${9}$&${7}$&${8}$&${13}$&${6}$\\
${{\bf{\alpha}}=1.8}$&$2^{8}$&${126}$&${17}$&${17}$&${10}$&${17}$&${18}$&${11}$&${17}$&${17}$&${10}$&${20}$&${22}$&${13}$&${9 }$&${10}$&${7}$&${8}$&${15}$&${7}$\\
    ${}$&$2^{9 }$            &${238}$&${17}$&${17}$&${9 }$&${17}$&${19}$&${11}$&${17}$&${17}$&${9 }$&${21}$&${23}$&${14}$&${9 }$&${14}$&${7}$&${8}$&${17}$&${7}$\\
    ${}$&$2^{10}$            &${448}$&${17}$&${17}$&${9 }$&${18}$&${20}$&${12}$&${17}$&${17}$&${9 }$&${22}$&${24}$&${14}$&${11}$&${18}$&${7}$&${8}$&${21}$&${7}$\\
				\hline
				\hline
			\end{tabular}
		\end{center}
	\end{small}
	\label{tab:T1}
\end{table}

Table \ref{tab:T1} shows the number of iterations of the proposed multigrid methods and preconditioners for three different values of $\alpha$.

Regarding multigrid methods, the observations are as follows:
\begin{itemize}
\item The Galerkin approach is robust, even as V-cycle, in agreement with the previous theoretical analysis.
\item The geometric approach is robust only when $\alpha>1.5$.
\item The robustness of the geometric multigrid can be improved by using it as a preconditioner, as seen in the column $\widetilde{P}{V_{1}^{1}}$. In particular, it is even more robust than the band multigrid preconditioner $P_{s}{V}_{1}^{1}$.
\end{itemize}

When comparing the circulant preconditioners, the Strang preconditioner $P_{S}$ is preferable to the optimal Chan preconditioner $P_{C}$.

%--------------------------------------------------------------------
\subsection{2D Problems}\label{sect:2D}
Here we extend our study to two-dimensional variable coefficient RFDE, given by
 \begin{align}\label{eq:2D_RFDE}
			-{\,c}(x,y)\frac{\partial^{\alpha} u(x,y)\,}{\partial\vert x\vert^{\alpha}}-{\,e}(x,y)\frac{\partial^{\beta} u(x,y)\,}{\partial\vert y\vert^{\beta}} =m(x,y)\,,\quad (x,y)\in\Omega=[a_1,b_1]\times[a_2,b_2],
	\end{align}
	with boundary condition
	\begin{align}\label{eq:2D_BC_RFDE}
			u(x,y)\ =0,\quad (x,y)\in\partial\Omega,
	\end{align}
where ${c}(x,y),\,{e}(x,y)$ are non-negative diffusion coefficients, $m(x,y)$ is the source term, and \begin{math}\frac{\partial^{\gamma} u(x,y)\,}{\partial\vert x\vert^{\gamma}}\end{math}, \begin{math}\frac{\partial^{\gamma} u(x,y)\,}{\partial\vert y\vert^{\gamma}}\end{math} are the  Riesz fractional derivatives for $\gamma\in\{\alpha,\beta\},\alpha,\beta\in(\,1,2)\,$. In order to define the uniform partition of the spatial domain $\Omega$, we fix $M_1, M_2\in\mathbb{N}$ and  define
\begin{equation*}\label{3D_GRID}
\begin{array}{lllll}
x_j=a_1+jh_x,		&& h_x=\frac{b_1-a_1}{M_1+1}, 		&& j=0,...,M_1+1,  	\\
y_j=a_2+jh_y,	    && h_y=\frac{b_2-a_2}{M_2+1},  		&& j=0,...,M_2+1.
\end{array}
\end{equation*}
Let us now introduce the following $M$-dimensional vectors, with $M=M_{1}M_{2}$:
\begin{align*}
{\bf{u}}=&[\,u_{1,1},u_{2,1},\cdots,u_{M_{1},1},u_{1,2},u_{2,2},\cdots,u_{M_{1},2},\cdots,u_{1,M_{2}},u_{2,M_{2}},\cdots,u_{M_{1},M_{2}}]^T,\\
{\bf{c}}=&[\,{c}_{1,1},{c}_{2,1},\cdots,{c}_{M_{1},1},{c}_{1,2},{c}_{2,2},\cdots,{c}_{M_{1},2},\cdots,{c}_{1,M_{2}},{c}_{2,M_{2}},\cdots,{c}_{M_{1},M_{2}}]^T,\\
{\bf{e}}=&[\,{e}_{1,1},{e}_{2,1},\cdots,{e}_{M_{1},1},{e}_{1,2},{e}_{2,2},\cdots,{e}_{M_{1},2},\cdots,{e}_{1,M_{2}},{e}_{2,M_{2}},\cdots,{e}_{M_{1},M_{2}}]^T,\\
{\bf{m}}=&[{m}_{1,1},{m}_{2,1},\cdots,{m}_{M_{1},1},{m}_{1,2},{m}_{2,2},\cdots,{m}_{M_{1},2},\cdots,{m}_{1,M_{2}},{m}_{2,M_{2}},\cdots,{m}_{M_{1},M_{2}}]^T,
\end{align*}
where $u_{i,j}=u(x_{i},y_{j})$ $c_{i,j}=c(x_{i},y_{j})$, and $e_{i,j}=e(x_{i},y_{j})$.

Using the shifted Gr\"unwald difference formula for both fractional derivatives in $x$ and $y$ leads to the matrix-vector form
	\begin{align}\label{eq: 2D_Final_Linear_sys}
		{A^{(\alpha,\beta)}_{M}}{\bf{u}}={\bf{m}},
	\end{align}
with the following (diagonal times two-level Toeplitz structured) coefficient matrix
\begin{align}\label{eq: 2D_coefficient_matrix}
        {A^{(\alpha,\beta)}_{M}}=\overline{c}_{\alpha}\bigl[C\bigl({I_{M_{2}}}\otimes \bigl({\mathcal{G}^{\alpha}_{M_{1}}}+ ({\mathcal{G}^{\alpha}_{M_{1}}})^{T}\bigr)\bigr)\bigr]+\overline{c}_{\beta}\bigl[E\bigl(\bigl({\mathcal{G}^{\beta}_{M_{2}}}+({\mathcal{G}^{\beta}_{M_{2}}})^{T}\bigr)\otimes I_{M_{1}}\bigr)\bigr],
\end{align}
where $C={\rm diag}({\bf c})$, $E={\rm diag}({\bf e})$, $\overline{c}_{\alpha}=\frac{c(\alpha)}{h^{\alpha}_{x}}$, $\overline{c}_{\beta}=\frac{c(\beta)}{h^{\beta}_{y}}$ and $c(\gamma)$, ${\mathcal{G}^{\gamma}_{M_{i}}}$ is defined as in Section \ref{sec:setting}.

Note that, the coefficient matrix ${A^{(\alpha,\beta)}_{M}}$ is strictly diagonally dominant M-matrix. Moreover, when the diffusion coefficients are constant and equal, the coefficient matrix ${A^{(\alpha,\beta)}_{M}}$ is a symmetric positive definite Block-Toeplitz with Toeplitz Blocks (\mbox{BTTB}) matrix.

% The generating function of ${A^{(\alpha,\beta)}_{M}}$ follows directly from the 1D case and is defined as
% \begin{equation*}
% \mathcal{F}_{\alpha,\beta}(x,y)=\mathcal{C_{\alpha}} \, f_{\alpha}(\,x)\, +\mathcal{E_{\beta}} \, f_{\beta}(\,y)\,,
% \end{equation*}
% with $\mathcal{C_{\alpha}}=cost\,\overline{c}_{\alpha}$ and $\mathcal{E_{\beta}}=cost\,\overline{c}_{\beta}$. Note that, since the generating functions $f_{\alpha}(x)\,$ and $f_{\beta}(y)\,$ are non-negative
% %which evident that the coefficient matrices ${A^{\alpha}_{M_{1}}}$ and ${A^{\beta}_{M_{2}}}$
% it follows that $\mathcal{F}_{\alpha,\beta}(x,y)$ is non-negative and hence ${A^{(\alpha,\beta)}_{M}}$ is symmetric positive definite.
In the following examples, we compare the performance of the MGMs, banded, circulant, and  $\tau$-based preconditioners. Precisely,
\begin{itemize}
\item Multigrid preconditioners, $P{V_{1}^{1}}$ and ${P}{\widetilde{V}_{1}^{1}}$. In both cases, the grid transfer operator is a bilinear interpolation, the weighted Jacobi is used as smoother, and one iteration of the V-cycle is performed to approximate the inverse of ${A^{(\alpha,\beta)}_{M}}$. Following the 1D case, the relaxation parameter of Jacobi $\omega^\star$ is computed as
\begin{equation}\label{eq:w2D_RFDE}
\omega^\star=\frac{4}{5}\zeta,
\end{equation}
where $\zeta=\frac{2\hat{\mathcal{F}}_0}{\|\mathcal{F}_{\alpha,\beta}(x,y)\|_\infty}$, with $\hat{\mathcal{F}}_0$ being the first Fourier coefficient of $\mathcal{F}_{\alpha,\beta}(x,y)$, the symbol of $A^{(\alpha,\beta)}_{M}$ (see \cite{moghaderi2017spectral} for more details).% {\color{red}$\mathcal{F}_{\alpha,\beta}(x,y)$, the symbol of $A^{(\alpha,\beta)}_{M}$ (see \cite{moghaderi2017spectral} for more details).}
\item According to the 1D numerical results, the Strang circulant preconditioner outperforms the optimal Chan preconditioner. Therefore, in the following, we consider only the 2-level Strang circulant preconditioner  defined as
        \begin{align}\label{eq:PS2D}
        {P_{S}}&=\overline{c}_{\alpha}{c^{av}}\big(I_{M_{2}}\otimes \bigr(S({\mathcal{G}^{\alpha}_{M_{1}}})+{S}({\mathcal{G}^{\alpha}_{M_{1}}})^{T}\big)\big)+\overline{c}_{\beta}{e^{av}}\big( \big({S}({\mathcal{G}^{\beta}_{M_{2}}})+{S}({\mathcal{G}^{\beta}_{M_{2}}})^{T}\big)\otimes I_{M_{1}}\big),
    \end{align}
    where ${c^{av}}=\mbox{mean}(\bf{c})$, ${e^{av}}=\mbox{mean}(\bf{e})$, and ${S}(T_M)$ is the Strang circulant approximation of the Toeplitz matrix $T_M$.
    \item Like in the 1D case, here we extend the banded preconditioner $P{_{s}V_{1}^{1}}$ strategy to the 2D case. To approximate the inverse of $_{s}\Tilde{A}^{(\alpha,\beta)}_{M}$, we performed $V_{1}^{1}$ iterations of V-cycle with Galerkin approach, weighted Jacobi as smoother and bilinear interpolation as grid transfer operator.

    The 2D-banded matrix is defined as
    \begin{align*}
    _{s}\Tilde{A}^{(\alpha,\beta)}_{M}&=\overline{c}_{\alpha}C\bigl({I_{M_{2}}}\otimes\bigl( {_{s}\mathcal{G}^{\alpha}_{M_{1}}}+ (_{s}{\mathcal{G}^{\alpha}_{M_{1}}})^{T}\bigr)\bigr) + \overline{c}_{\beta} E\bigl(\bigl(_{s}{\mathcal{G}^{\beta}_{M_{2}}} + (_{s}{\mathcal{G}^{\beta}_{M_{2}}})^{T}\bigr)\otimes I_{M_{1}}\bigr),
    \end{align*}
    \item $P_{\bf\tau}$ denotes the $\tau$ based preconditioner. For the symmetric Toeplitz matrix $\mathcal{\bf G}_{M_{i}}^{\left(\gamma\right)}= {_{s}\mathcal{G}^{\gamma}_{M_{i}}}+({_{s}\mathcal{G}^{\gamma}_{M_{i}}})^T$, with $i=1,2$, the $\tau$ matrix is
\begin{align}\label{eq: tau_matrix}
    \tau(\mathcal{\bf G}_{M_{i}}^{\left(\gamma\right)})=\mathcal{\bf G}_{M_{i}}^{\left(\gamma\right)}-H_{M_{i}},
\end{align}
where $H_{M_{i}}$ is a Hankel matrix, whose entries along antidiagonals are constant and equal to
\begin{align*}
        -[{g_{3}^{(\gamma)\,}},\,{g_{4}^{(\gamma)\,}},\,\dots,\,{g_{M_{i}}^{(\gamma)\,}},\,0,\,0,\,0,\,{g_{M_{i}}^{(\gamma)\,}},\,\dots,{g_{4}^{(\gamma)\,}},\,{g_{3}^{(\gamma)\,}}].
\end{align*}
Thanks to \cite{bini1983spectral}, the $\tau$ matrix can be diagonalized as follows
\begin{align}\label{eq: tau_matrix_diagonalization}
    \tau(\mathcal{\bf G}_{M_{i}}^{\left(\gamma\right)})=S_{M_{i}}D_{M_{i}}S_{M_{i}},
\end{align}
where $D_{M_{i}}=\text{diag}(\lambda(\mathcal{\bf G}_{M_{i}}^{\left(\gamma\right)}-H_{M_{i}}))$ and $S_{M_{i}}$ is the sine transform matrix whose entries are
\begin{equation}\label{eq: sine_transform_matrix}
    S_{M_{i}}=\left[\sqrt{\frac{2}{M_{i}+1}}\sin{\bigg(\frac{ij\pi}{M_{i}+1}\bigg)}\right]_{i,j=1}^{M_{i}}.
\end{equation}
The $\tau$ based preconditioner of the resulting linear system \eqref{eq: 2D_Final_Linear_sys} is
\begin{align}\label{eq: tau_based_preconditioner}
    P_{\bf\tau}=\mathcal{D}_{M}{\bf{S}}_{M}{\bf D}_{M}{\bf{S}}_{M},
\end{align}
with ${\bf{S}}_{M}=S_{M_{1}}\otimes S_{M_{2}}$, ${\bf D}_{M}=\big({I_{M_{2}}}\otimes D_{M_{1}}+D_{M_{2}}\otimes {I_{M_{1}}}$\big) and $\mathcal{D}_{M}$ is the diagonal matrix that contains the average of the diffusion coefficients. Regarding the computational point of view, $P_{\bf\tau}$ is extremely suitable because the matrix-vector product can be performed in $O(M\log M)$ operations through the discrete sine transform (DST) algorithm.

%In Table \ref{tab:T22}, we also report the results provided by ${V_{1\tau}^{1}}$, where we use one iteration of $P_{\bf\tau}$-preconditioned GMRES as a pre-smoother on the finest level and then Jacobi as pre and post-smoother respectively.
\end{itemize}

\paragraph{Example 2.} In this example, we consider the following 2D-RFDE given in equation \eqref{eq:2D_RFDE}, with $c(x,y)=e(x,y)=1$ and $\Omega=[\,0,1\,]\times[\,0,1\,]$. The source term and exact solution are given as
\begin{equation}\label{eq:2D_Exact1}
    u(x,y)=x^{2}(\,1-x\,)^{2}y^{2}(\,1-y\,)^{2},
\end{equation}
and
\begin{align*}\label{eq:2D_source1}
    m(x,y)=&\frac{1}{\cos(\frac{\alpha\pi}{2})}y^{2}(\,1-y\,)^{2}\biggl[\frac{2}{\Gamma(3-\alpha)}\biggl(x^{2-\alpha}+(1-x)^{2-\alpha}\biggr)-\frac{12}{\Gamma(4-\alpha)}\biggl(x^{3-\alpha}+(1-x)^{3-\alpha}\biggr)\nonumber\\
    &+\frac{24}{\Gamma(5-\alpha)}\biggl(x^{4-\alpha}+(1-x)^{4-\alpha}\biggr)\biggr]+\frac{1}{\cos(\frac{\beta\pi}{2})}x^{2}(\,1-x\,)^{2}\biggl[\frac{2}{\Gamma(3-\beta)}\biggl(y^{2-\beta}+(1-y)^{2-\beta}\biggr)\nonumber\\
    &-\frac{12}{\Gamma(4-\beta)}\biggl(y^{3-\beta}+(1-y)^{3-\beta}\biggr)+\frac{24}{\Gamma(5-\beta)}\biggl(y^{4-\beta}+(1-y)^{4-\beta}\biggr)\biggr].
\end{align*}

\begin{table}
	\caption{Example 2: Number of iterations for different values of $\alpha$ and $\beta$ with $M_{1}=M_{2}$.}
	\begin{small}
		\setlength{\tabcolsep}{5pt}
		\begin{center}
			\begin{tabular}{c c c c c c c c c c c c c c c c c c}
		%	\hline
				\hline
				%%%%%%%%%%%%%%%%%
				\noalign{\vskip 1mm}
				\multirow{1}{*} {}&
				\multicolumn{2}{c}{${}$} &
                \multicolumn{2}{c}{${\mbox{\bf Galerkin}}$} &
				\multicolumn{2}{c}{${\mbox{\bf Geometric}}$} &
                \multicolumn{10}{c}{$\mbox{\bf Preconditioners}$} \\
				\cmidrule(r){4-5}\cmidrule(r){6-7}\cmidrule{8-17}
				$(\alpha,\beta)$&$M_{1}+1$& ${\mbox{CG}}$& ${V_{1}^{1}}$& ${\mbox{T(s)}}$& ${V_{1}^{1}}$ &${\mbox{T(s)}}$& ${P}{V_{1}^{1}}$&${\mbox{T(s)}}$& $\widetilde{P}{V_{1}^{1}}$&${\mbox{T(s)}}$& ${P{_{s}{V}_{1}^{1}}}$&${\mbox{T(s)}}$& $P_{\tau}$&$\mbox{T(s)}$& ${P_{S}}$&${\mbox{T(s)}}$\\ \hline
				%%%%%%%%%%%%%%%%%%%%%%%%%%%%%%%
	${         }$       &$2^5$   &${57 }$&${17}$&${0.036}$&${36}$&${0.134}$&${9 }$&${0.083}$&${13}$&${0.144}$&${10}$&${0.085}$&${6}$&${0.091}$&${13}$&${0.085}$\\
	${(1.1,1.2)}$       &$2^6$   &${93 }$&${14}$&${0.126}$&${43}$&${0.322}$&${9 }$&${0.167}$&${14}$&${0.229}$&${11}$&${0.121}$&${7}$&${0.100}$&${17}$&${0.114}$\\
 ${\omega^\star=0.83}$&$2^7$   &${157}$&${14}$&${0.809}$&${48}$&${0.815}$&${8 }$&${0.604}$&${15}$&${0.465}$&${12}$&${0.325}$&${7}$&${0.144}$&${19}$&${0.297}$\\
    ${         }$       &$2^8$   &${237}$&${14}$&${5.201}$&${52}$&${3.085}$&${8 }$&${3.701}$&${16}$&${1.587}$&${17}$&${1.451}$&${8}$&${0.424}$&${21}$&${0.604}$\\
    ${         }$       &$2^9$   &${383}$&${14}$&${46.42}$&${56}$&${11.09}$&${9 }$&${60.46}$&${17}$&${5.645}$&${26}$&${38.62}$&${8}$&${1.055}$&${24}$&${2.002}$\\
    \hline
    ${         }$       &$2^5$   &${44 }$&${14}$&${0.035}$&${19}$&${0.083}$&${8 }$&${0.082}$&${8 }$&${0.127}$&${9 }$&${0.079}$&${6}$&${0.088}$&${12}$&${0.081}$\\
	${(1.5,1.5)}$       &$2^6$   &${78 }$&${14}$&${0.121}$&${21}$&${0.194}$&${8 }$&${0.151}$&${9 }$&${0.178}$&${10}$&${0.119}$&${6}$&${0.102}$&${13}$&${0.111}$\\
    ${\omega^\star=0.85}$&$2^7$&${136}$&${12}$&${0.737}$&${23}$&${0.418}$&${8 }$&${0.592}$&${10}$&${0.331}$&${12}$&${0.308}$&${7}$&${0.143}$&${16}$&${0.286}$\\
    ${         }$       &$2^8$   &${234}$&${13}$&${4.843}$&${25}$&${1.625}$&${8 }$&${3.313}$&${10}$&${1.025}$&${15}$&${1.424}$&${8}$&${0.416}$&${20}$&${0.610}$\\
    ${         }$       &$2^9$   &${401}$&${13}$&${43.68}$&${26}$&${5.710}$&${8 }$&${56.52}$&${11}$&${3.811}$&${25}$&${34.69}$&${8}$&${1.071}$&${25}$&${2.057}$\\
    \hline
    ${         }$       &$2^5$   &${66 }$&${24}$&${0.041}$&${26}$&${0.114}$&${11}$&${0.089}$&${11}$&${0.139}$&${11}$&${0.085}$&${6}$&${0.091}$&${15}$&${0.089}$\\
    ${(1.7,1.9)}$       &$2^6$   &${127}$&${27}$&${0.205}$&${30}$&${0.235}$&${12}$&${0.181}$&${12}$&${0.215}$&${13}$&${0.128}$&${6}$&${0.104}$&${19}$&${0.118}$\\
    ${\omega^\star=0.83}$&$2^7$&${244}$&${30}$&${1.557}$&${34}$&${0.598}$&${13}$&${0.904}$&${13}$&${0.407}$&${15}$&${0.385}$&${6}$&${0.140}$&${25}$&${0.383}$\\
    ${         }$       &$2^8$   &${467}$&${33}$&${11.62}$&${38}$&${2.350}$&${14}$&${5.616}$&${15}$&${1.405}$&${17}$&${1.503}$&${7}$&${0.401}$&${30}$&${0.684}$\\
    ${         }$       &$2^9$   &${899}$&${37}$&${124.3}$&${43}$&${8.658}$&${15}$&${87.13}$&${16}$&${5.464}$&${27}$&${40.18}$&${7}$&${1.103}$&${43}$&${3.594}$\\
    \hline
			\end{tabular}
		\end{center}
	\end{small}
	\label{tab:T21}
\end{table}

Table \ref{tab:T21} presents the number of iterations and CPU times required by the proposed multigrid methods and preconditioners. The $\tau$ preconditioner $P_{\tau}$ stands out as the most robust choice, both in terms of iteration number and CPU time. The geometric multigrid used as a preconditioner is also a good option, although the robustness of multigrid methods declines in the case of anisotropic problems, i.e., when $(\alpha, \beta) = (1.7, 1.9)$. In such a case, it is advisable to adopt the strategies proposed in \cite{donatelli2020multigrid}.

\paragraph{Example 3.} This example is taken from \cite{pan2021efficient}, with $\Omega=[\,0,2\,]\times[\,0,2\,]$. The diffusion coefficients are
\begin{equation}\label{eq:2D_Exp2_coef}
    {c}(x,y)=1,\quad {e}(x,y)=1+xy,
\end{equation}
and the source term is built from the exact solution given as
\begin{equation}\label{eq:2D_Exact2}
    u(x,y)=x^{4}(\,2-x\,)^{4}y^{4}(\,2-y\,)^{4}.
\end{equation}

Due to the nonsymmetry of the coefficient matrix caused by diffusion coefficients, the CG method has been replaced with GMRES. Nevertheless, the numerical results in Table \ref{tab:T22M} are comparable to those in Table \ref{tab:T21}. Indeed, the $\tau$ preconditioner remains the best option, and the geometric multigrid preconditioner provides a good and robust alternative, especially when the problem is isotropic, i.e., when $\alpha \approx \beta$, also when applied to the band approximation.

\begin{table}
	\caption{Example 3: Number of iterations for different values of $\alpha$ and $\beta$ with $M_{1}=M_{2}$.}
	\begin{small}
		\setlength{\tabcolsep}{2.5pt}
		\begin{center}
			\begin{tabular}{c @{\qquad} c @{\qquad} c @{\qquad}c c @{\qquad}c c @{\qquad}c c @{\qquad} c c}
	%		\hline
%			\hline
				%%%%%%%%%%%%%%%%%
				$(\alpha,\beta)$&$M_{1}+1$& ${\mbox{GMRES}}$& $\widetilde{P}{V_{1}^{1}}$& ${\mbox{T(s)}}$& ${P{_{s}{V}_{1}^{1}}}$&${\mbox{T(s)}}$& ${P_{\tau}}$&${\mbox{T(s)}}$& ${P_{S}}$&${\mbox{T(s)}}$\\ \hline
				%%%%%%%%%%%%%%%%%%%%%%%%%%%%%%%
	${         }$      &$2^4$   &${49 }$&${11}$&${0.147}$&${9 }$&${0.100}$&${12}$&${0.114}$&${19}$&${0.118}$\\
	${(1.1,1.2)}$      &$2^5$   &${94 }$&${12}$&${0.181}$&${11}$&${0.110}$&${13}$&${0.135}$&${23}$&${0.140}$\\
${\omega^\star=0.83}$&$2^6$   &${163}$&${14}$&${0.258}$&${12}$&${0.159}$&${13}$&${0.160}$&${27}$&${0.201}$\\
    ${         }$      &$2^7$   &${267}$&${15}$&${0.537}$&${15}$&${0.407}$&${13}$&${0.250}$&${31}$&${0.441}$\\
    \hline
    ${         }$      &$2^4$   &${51 }$&${11}$&${0.147}$&${10}$&${0.101}$&${13}$&${0.118}$&${20}$&${0.120}$\\
	${(1.5,1.5)}$      &$2^5$   &${95 }$&${12}$&${0.182}$&${12}$&${0.112}$&${13}$&${0.133}$&${24}$&${0.146}$\\
${\omega^\star=0.85}$&$2^6$   &${171}$&${12}$&${0.240}$&${13}$&${0.164}$&${14}$&${0.168}$&${27}$&${0.196}$\\
    ${         }$      &$2^7$   &${301}$&${12}$&${0.531}$&${17}$&${0.412}$&${14}$&${0.261}$&${31}$&${0.423}$\\
    \hline
    ${         }$      &$2^4$   &${62 }$&${13}$&${0.146}$&${12}$&${0.103}$&${12}$&${0.116}$&${21}$&${0.122}$\\
    ${(1.7,1.9)}$      &$2^5$   &${129}$&${14}$&${0.189}$&${14}$&${0.115}$&${13}$&${0.134}$&${26}$&${0.150}$\\
${\omega^\star=0.83}$&$2^6$   &${262}$&${15}$&${0.281}$&${14}$&${0.164}$&${14}$&${0.167}$&${31}$&${0.214}$\\
    ${         }$      &$2^7$   &${524}$&${15}$&${0.544}$&${18}$&${0.425}$&${14}$&${0.253}$&${37}$&${0.483}$\\
     \hline
     ${         }$     &$2^4$   &${60 }$&${11}$&${0.144}$&${11}$&${0.103}$&${12}$&${0.144}$&${21}$&${0.121}$\\
    ${(1.9,1.9)}$      &$2^5$   &${124}$&${12}$&${0.183}$&${11}$&${0.109}$&${13}$&${0.132}$&${25}$&${0.148}$\\
${\omega^\star=0.82}$&$2^6$   &${252}$&${12}$&${0.249}$&${11}$&${0.153}$&${13}$&${0.164}$&${30}$&${0.208}$\\
     ${         }$     &$2^7$   &${504}$&${12}$&${0.444}$&${13}$&${0.388}$&${13}$&${0.247}$&${35}$&${0.445}$\\
    \hline
			\end{tabular}
		\end{center}
	\end{small}
	\label{tab:T22M}
\end{table}
\section{Application to image deblurring}\label{sec:imdeblur}

Fractional differential operators have been investigated to enhance diffusion in image denoising and deblurring  \cite{yang2016fractional,antil2017spectral,aleotti2023fractional}. In this section, we investigate the effectiveness of the previous preconditioning strategies to the Tikhonov regularization problem
\begin{equation}\label{eq:tik}
   \min_{\v{u}\in\R^M} \|B_M \v{u}-\v{m}\|_2^2 + \mu \v{u}^T A_M^{(\alpha,\beta)}\v{u},
\end{equation}
where $\v{m}$ is the noisy and blurred observed image, $B_M$ is the discrete convolution operator associated with the blurring phenomenon, and $\mu>0$ is the regularization parameter.

The solution of the least square problem \eqref{eq:tik} can be obtained by solving the linear system
\begin{equation}\label{eq:sistik}
   \left(B_M^TB_M + \mu A_M^{(\alpha,\beta)}\right)\v{u}=\v{m},
\end{equation}
by the preconditioned CG method since the coefficient matrix is positive definite.
The matrix $B_M$ exhibits a block Toeplitz with Toeplitz blocks structure as the matrix $A_M^{(\alpha,\beta)}$, but the spectral behavior is completely different since the ill-conditioned subspace is very large and intersects substantially the high frequencies. Therefore, applying multigrid methods is quite challenging, see \cite{donatelli2005multigrid}, and requires further investigation in the future. Here, we show the effectiveness of the $\tau$ preconditioner in comparison with the Strang circulant preconditioner.

We consider the true \textsf{satellite} image of $128 \times 128$ pixels in Figure \ref{fig:imm} (a). The observed image in Figure \ref{fig:imm} (b) is affected by a Gaussian blur and $5\%$ of white Gaussian noise. We choose the fractional derivative $\alpha=\beta=1.1$ and the diffusion coefficients $c(x,y)=e(x,y)=1$. The matrix $A_M^{(\alpha,\beta)}$ is defined as in \eqref{eq: 2D_coefficient_matrix}, while the Strang circulant and the $\tau$ preconditioners are defined in \eqref{eq:PS2D} and \eqref{eq: tau_based_preconditioner}, respectively.
Since the quality of the reconstruction depends on the model \eqref{eq:tik}, all the different solvers provide a similar reconstruction. The reconstructed image for $\mu=10^{-4}$ is shown in Figure \ref{fig:imm} (c).

Various strategies exist for estimating the parameter $\mu$, see e.g. \cite{hansen1998rank}, but such estimation is beyond the scope of this work. Instead, we test different values of $\mu$. The linear system \eqref{eq:tik} is solved using the built-in \textit{pcg} Matlab function with a tolerance of $10^{-6}$.
Table \ref{tab:imm} reports the number of iterations and the related restoration error for various values of $\mu$.
It is noteworthy that the $\tau$ preconditioner demonstrates a significant speedup compared to both CG without preconditioning and PCG with the Strang circulant preconditioner for all relevant values of $\mu \in [10^{-5}, 10^{-3}]$.

We should also point out that the considered images show a black background and hence all the various BCs are equivalent in terms of the precision of the reconstruction. However, in a general setting, we observe that the most precise BCs are related to the $\tau$ algebra (see \cite{tau-AR-BCs-1,tau-AR-BCs-2} and references therein).

  \begin{figure}
		\centering
		%\begin{center}
		\begin{subfloat}[]
			{\resizebox*{4cm}{!}     {\includegraphics[width=\textwidth]{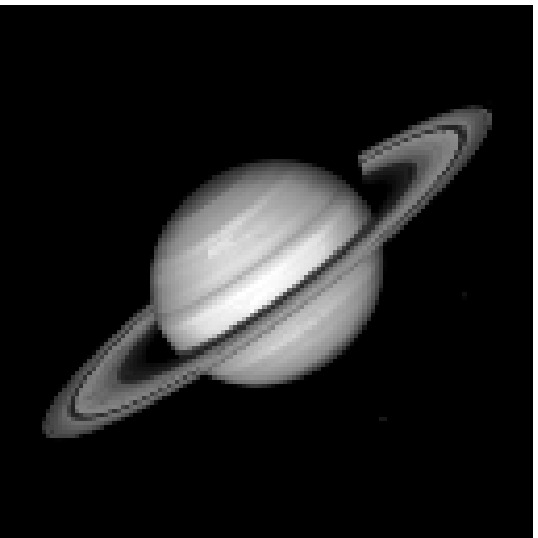}}}
		\end{subfloat}
              \hspace{0.5cm}
             \begin{subfloat}[]
			{\resizebox*{4cm}{!}{\includegraphics[width=\textwidth]{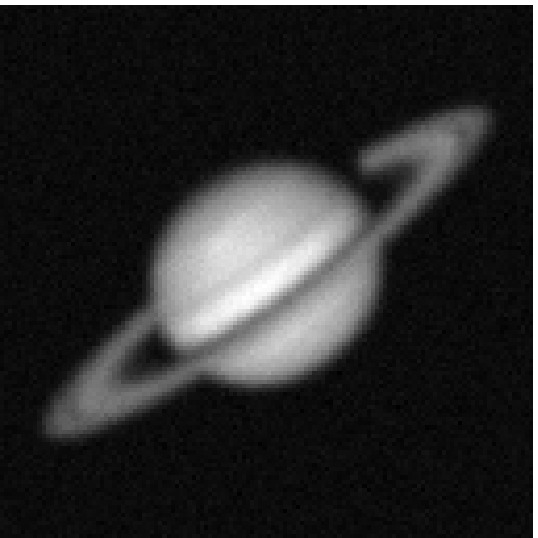}}}
		\end{subfloat}
              \hspace{0.5cm}
            \begin{subfloat}[]
			{\resizebox*{4cm}{!}{\includegraphics[width=\textwidth]{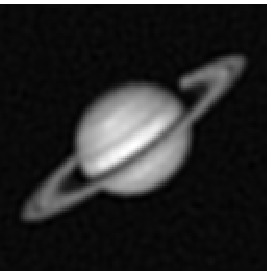}}}
		\end{subfloat}

		\caption{Image deblurring example: (a) true image, (b) observed image, (c) restored image for $\mu=10^{-4}$.}
	\label{fig:imm}
\end{figure}

\begin{table}
	\caption{Image deblurring example: number of iterations and relative restoration error (RRE) for different values of $\mu$.}
	\begin{small}
		\setlength{\tabcolsep}{2.5pt}
		\begin{center}
			\begin{tabular}{c @{\qquad} c @{\qquad}c  @{\qquad}c @{\qquad}c}
$\mu$& ${\mbox{CG}}$&  ${P_{\tau}}$& ${P_{S}}$ & RRE\\
\hline
$10^{-3}$ & 12 & 3 & 12 & 1.54e-01\\
$10^{-4}$ & 15 & 4 & 9 & 1.12e-01\\
$10^{-5}$ & 36 & 6 & 16 & 1.15e-01\\
$10^{-6}$ & 93 & 10 & 44 & 2.21e-01\\
\hline
				%%%%%%%%%%%%%%%%%%%%%%%%%%%%%%%
			\end{tabular}
		\end{center}
	\end{small}
	\label{tab:imm}
\end{table}

%-----------------
\section{Conclusions}\label{sec:end}
We discussed the convergence analysis of the multigrid method applied to a Riesz problem and compared it with state of-the-art techniques. Associated numerical experiments highlight that $\tau$ preconditioning is the best performing method. An application of the latter to an image deblurring problem with Tikhonov regularization is given. This specific application will be considered in future research works, especially considering numerical methods for nonlinear models that require solving a linear problem as an inner step, see e.g. \cite{cai2016regularization}.

%Regarding theoretical results we have two cases:
%\begin{enumerate}
%    \item  If $p_{k,\alpha}(x)=\sqrt{2}\,p(x)$, in this case the $\|f_{k,\alpha}\|_{\infty}$ goes to $\infty$ at the coarsest level (See \figurename~\ref{fig: Plot at different level}). Unfortunately, the value of the Ruge-Stuben parameter is $\sigma_{k}=\frac{a^{(k)}_{0}}{\|f_{k,\alpha}\|_{\infty}}$, then $\sigma_{k}$ goes to zero as the parameter $k$ diverges.
%    \item
%If $p_{k,\alpha}(x)=C_{k+1,\alpha}\,p(x)$, then in this case at each level $\|f_{k+1,\alpha}\|_{\infty}=\|f_{k,\alpha}\|_{\infty}=2^{\alpha+1}$ (see \figurename~\ref{fig:fk}).
%\end{enumerate}

%The related issues will be taken into account in future research works.

%{\bf *{Plot of the limit functions (k large enough), for different values of $\alpha$:}}
%\begin{figure}[H]
%		\centering
%			\includegraphics[width=0.7\textwidth]{Plots.jpg}
%		\caption{Plots of $f_{20,\alpha_{i}}(\,x)\,$ for different values of $\alpha$.}
%	\label{fig: Plots}
%	\end{figure}
%%%%%%%%%%%
%%%%%%%%%%%

\section*{Acknowledgements}
The work of the authors is supported by the GNCS-INdAM project CUP E53C22001930001 and project  CUPE53C23001670001.
The work of the second author is partially funded by MUR -- PRIN 2022, grant number 2022ANC8HL.
The third author
acknowledges the MUR Excellence Department Project MatMod@TOV awarded to the Department of Mathematics, University of
Rome Tor Vergata, CUP E83C23000330006.
The work of the fourth author is funded by the European High-Performance Computing Joint Undertaking  (JU) under grant agreement No 955701. The JU receives support from the European Union’s Horizon 2020 research and innovation program and Belgium, France, Germany, and Switzerland.
Furthermore, the fourth is grateful for the support of the Laboratory of Theory, Economics and Systems – Department of Computer Science at Athens University of Economics and Business.

\bibliography{Bibtex}
\bibliographystyle{unsrt}
\end{document}